\documentclass[onefignum,onetabnum]{siamart251216}
\usepackage{enumitem}
\usepackage{amsfonts}
\usepackage{latexsym} 
\usepackage{float}
\usepackage{geometry}
\usepackage[latin9]{inputenc}
\usepackage{amstext}
\usepackage{amsmath}
\usepackage{amssymb}
\usepackage{graphicx}
\usepackage{xcolor}
\usepackage{booktabs}
\usepackage{caption}

\makeatletter
\theoremstyle{plain}
\newtheorem{thm}{Theorem}
\theoremstyle{plain}
\newtheorem{lem}[thm]{Lemma}
\newtheorem{assumption}{Assumption}
\usepackage{rotating}
\usepackage{epsfig}
\usepackage{subfigure}
\usepackage{url}
\usepackage{diagbox}
\usepackage{booktabs}
\newcommand{\tnorm}[1]{%
  \left|\!\left|\!\left| #1 \right|\!\right|\!\right|
}

\makeatother

\title{A Two-Level Preconditioner Based on Dominant Components for Time-Dependent Multiscale High-Contrast Problem}

\begin{document}

 \author{Yating Wang \thanks{School of Mathematics and Statistic, Xi'an Jiaotong University, Xi'an, People's Republic of China. (\email{yatingwang@xjtu.edu.cn}).}
 \and  Yibao Li \thanks{School of Mathematics and Statistic, Xi'an Jiaotong University, Xi'an, People's Republic of China. (\email{yibaoli@xjtu.edu.cn}).}
 \and Wing Tat Leung \thanks{Department of Mathematics, City University of Hong Kong, Hong Kong Special Administrative Region. (\email{wtleung27@cityu.edu.hk}} )} 
\maketitle

\begin{abstract}
In this work, we develop a two-level overlapping preconditioner for time-dependent problems in high-contrast multiscale media. We present a coarse-space construction based on multiscale methods, with emphasis on the relaxed nonlocal multicontinuum (NLMC) method. The NLMC space can be separated into components representing the high-permeability regions and the low-permeability background. We show that the complete NLMC space effectively preconditions the heterogeneous stiffness operator, whereas the high-permeability component alone captures the contrast-dependent global modes. For suitable small time-step sizes, the mass matrix controls the low-permeability contribution and only the high-permeability component in NLMC space is required for the global coarse correction in the two-level preconditioner. For general time-step sizes, performance can be maintained by using the full NLMC space or augmenting the high-permeability component with a standard multiscale space. The proposed coarse-space construction lowers the computational cost while preserving robustness with respect to coefficient contrast and fine-scale resolution. To further improve efficiency, the multiscale basis functions can be constructed by iteratively solving the relaxed energy-minimizing formulation. We demonstrate the robustness, efficiency and scalability of the proposed method through several numerical experiments.
\end{abstract}

\section{Introduction}
Subsurface flow through porous media and heat transfer in heterogeneous materials are fundamental processes in a wide range of scientific and engineering applications, including reservoir simulation, groundwater flow and contaminant transport, and thermal transport in composite or geological media. A common feature of these problems is the coexistence of multiple spatial scales. Material parameters such as permeability and thermal conductivity may vary rapidly within the computational domain, differ by several orders of magnitude, and form complicated connected structures such as fractures, inclusions, and long conductive channels. Accurate numerical simulation therefore requires a fine mesh capable of resolving the relevant heterogeneities. The resulting discrete systems are typically very large and severely ill-conditioned, with their conditioning deteriorating under both mesh refinement and increasing coefficient contrast. In particular, highly conductive channels may generate solution components that cannot be represented adequately by conventional coarse-scale polynomial functions. These features make the design of accurate, robust, and computationally scalable solvers especially challenging for high-contrast multiscale problems \cite{ye2024robust, ye2025highly, GalvisChungEfendievLeung2018}.

A broad class of multiscale model-reduction methods has been developed to address these difficulties. Numerical homogenization and upscaling methods replace unresolved heterogeneities by effective coarse-scale parameters, thereby reducing the dimension of the original problem. The multiscale finite element method (MsFEM) instead constructs local basis functions by solving cell problems that incorporate fine-scale coefficient information \cite{HouWu1997}. The generalized multiscale finite element method (GMsFEM) extends this idea through local snapshot spaces and spectral decompositions, from which a small number of dominant modes are selected to form the multiscale approximation space \cite{EfendievGalvisHou2013}. In high-contrast media, the eigenfunctions associated with small local eigenvalues are particularly important because they represent fractures, channels, and other nonlocal features that cannot be localized by standard coarse functions. Constraint energy-minimizing GMsFEM (CEM-GMsFEM) further combines such an auxiliary spectral space with constrained or relaxed energy minimization over oversampled regions. When the relevant channel modes are included in the auxiliary space, the resulting localized basis functions provide convergence that is robust with respect to the coefficient contrast \cite{ChungEfendievLeung2018CEM}. Nonlocal multicontinuum (NLMC) methods follow a related principle by introducing separate coarse variables for distinct physical continua, such as connected high-permeability networks and the surrounding low-permeability matrix, and by constructing energy-minimizing basis functions that encode their nonlocal interactions
\cite{ChungEfendievLeungVasilyevaWang2018}. Although their constructions differ, these methods share the central strategy of embedding heterogeneous fine-scale information into effective multiscale basis functions. In their original and most common model-reduction role, these basis functions span a reduced approximation space in which the coarse-scale problem is solved directly.

An alternative to solving a reduced problem is to retain the fine-grid discretization and accelerate the resulting algebraic system using a preconditioned iterative method \cite{Griebel1995,Dolean2015,Heinlein2018,Kim2017,Calvo2016,Sarkis1997,Wang2014,Saad2003,Li2020,Dolean2012,Graham2007,Nataf2010,Heinlein2019,Briggs2000}. This approach is particularly relevant for large-scale simulations, for which sparse direct solvers may require excessive memory and exhibit limited parallel scalability. In a two-level overlapping domain decomposition method, overlapping local solvers eliminate localized error components, while a global coarse solver communicates information across the entire computational domain. The construction of the coarse space is therefore important for the performance of the preconditioner. For high-contrast problems, classical coarse spaces composed of low-order polynomial functions generally do not capture the contrast-dependent near-kernel modes associated with conductive channels. Consequently, the condition number and iteration count may grow rapidly with the contrast, and the iterative solver may stagnate for sufficiently complicated coefficient fields. This observation has motivated the replacement of polynomial coarse spaces by nonstandard, heterogeneity-aware spaces constructed from MsFEM basis functions, GMsFEM eigenfunctions and local spectral modes
\cite{GrahamScheichl2007,GalvisEfendiev2010a,GalvisEfendiev2010b, YangFuChung2019,ye2024robust}. With a suitable selection of local spectral modes, two-level domain decomposition preconditioners can attain condition-number bounds that are independent of the coefficient contrast and unresolved spatial scales. The majority of the available coarse-space constructions and analyses, however, concern stationary elliptic operators or steady incompressible flow systems. The systematic construction of reduced multiscale coarse spaces for the fully implicit linear systems arising from time-dependent high-contrast problems
is comparatively less developed.

In this paper, we study two-level overlapping preconditioners for
time-dependent multiscale flow and heat-transfer problems which has the form
\begin{equation}
\frac{\partial u}{\partial t}
  - \nabla\cdot\bigl(\kappa \nabla u\bigr)
  = f,
  \label{eq:transient-model}
\end{equation}
where \(\kappa\) denotes a high-contrast permeability or thermal conductivity. After a fully
implicit time discretization, each time step requires the solution of a system
\begin{equation}
  \widetilde{A}u^{n+1}=b^{n+1},
  \qquad
  \widetilde{A}=M+\Delta t\,A,
  \label{eq:fully-implicit-system}
\end{equation}
where \(M\) and \(A\) are the fine-grid mass and stiffness matrices, respectively. We aim to construct a two-level overlapping additive preconditioner such that the prolongation operator is determined by a multiscale coarse space. The local corrections resolve fine-scale behavior within the overlapping subdomains, whereas the coarse correction captures the global components of the operator. The presence of the mass matrix makes it possible to use a substantially smaller coarse space than would be required for the stiffness matrix alone.

Our coarse-space construction is based on an NLMC-type decomposition. The auxiliary functions distinguish connected high-permeability or conductivity continua from the low-permeability background matrix. The resulting multiscale space is decomposed as \(V_H = V_{H,1}+V_{H,2}\), where \(V_{H,1}\) is generated by basis functions associated with high-permeability channels or fractures, while \(V_{H,2}\) represents the low-permeability matrix continua. The complete space \(V_H\) is suitable for preconditioning the heterogeneous stiffness operator \(A\). More importantly, \(V_{H,1}\) alone captures the contrast-dependent global modes generated by the high-permeability network.

The main observation of this work is that for a general choice of time step size, we can show $V_H = V_{H,1} + W_H$ is a good choice of coarse space to precondition the stiffness matrix, where $W_H$ can be \(V_{H,2}\) or other standard multiscale finite element space. When \(\Delta t \lesssim H^2 \), the mass matrix controls the low-permeability contribution of the stiffness matrix, the low-permeability component does not require an additional global coarse correction for the transient operator \(M+\Delta t\,A\). Consequently, the reduced choice \(V_{H,1}\) is sufficient for the two-level preconditioner. Thus, only the high-permeability component must be treated carefully by the global coarse solver. Thus when the volume of high permeability part in the coefficients is low, retaining only \(V_{H,1}\) reduces the dimension and computational cost of the global coarse problem while preserving robust convergence. This also indicates an important distinction between multiscale approximation and multiscale preconditioning. A basis function may be necessary for accurately representing every coarse-scale continuum in a direct reduced model, but it need not be included in a preconditioner when the remaining part of the discrete operator already provides sufficient control.

The layout of the paper is as follows. In Sec.~\ref{sec:relaxed_NLCM}, we will introduce the coarse space $V_0$ based on relaxed constraint energy minimization problem. In Sec.~\ref{sec:iterative_construction}, we will discuss iterative method for the coarse space construction. In Sec.~\ref{sec:preconditioner}, we will introduce a reduced dimension coarse space for the time dependent problems and analyze the proposed preconditioner. In Sec.~\ref{sec:numerical}, we will present some numerical examples to demonstrate the performance of the proposed preconditioning methods.

\section{Preliminary}
We will focus on the multiscale diffusion equation, and the methods and algorithms discussed in this paper can be applied to various PDEs.

Given a domain $\Omega \subset \mathbb{R}^d$ ($d=2,3$) and a source function $f\in L^2(\Omega)$, we seek the solution $u: \Omega \times (0,T] \to \mathbb{R}$ of the equation
\[
\cfrac{\partial u}{\partial t} -\nabla\cdot\left( \kappa \nabla u \right) = f \quad \text{in } \Omega,
\]
with zero Dirichlet boundary condition $u = 0$ on $\partial \Omega \times [0,T]$, and initial condition $u|_{\Omega\times \{0\}}=u_0$ with $u_0:\Omega\to \mathbb{R}$. The permeability field $\kappa$ is the permeability coefficient defined on the domain $\Omega$. Here, we consider $\kappa(x) \in L^{\infty}(\Omega)$ and has the high contrast property, that is, $0<\kappa_{\text{min}}\leq\kappa(x)\leq\kappa_{\text{max}}$
with the contrast ratio $\eta :=\cfrac{\kappa_{\text{max}}}{\kappa_{\text{min}}} \gg 1$. 

\subsection{The fully discretized scheme}
Let $\mathcal{T}^h$ denote a partition of $\Omega$ with mesh size $h$, and let $V_h = P_1(\mathcal{T}^h) \subset H_1(\Omega)$ be the finite element space consisting of piecewise linear functions, and $V_{h,0}=\{v\in V_h, \; v=0 \text{ on } \partial \Omega \}$. We note that both the mesh generation and the choice of fine-grid finite element space remain flexible in our proposed framework.

We then introduce a semi-discretization of the problem, that is to find $u_{h}(t,\cdot)\in V_{h,0}$,
such that 
\[
(\partial_{t}u_{h},v) + a(u_{h},v)=(f,v), \quad\forall v\in V_{h,0},
\]
where $a(u,v) : = \int_{\Omega} \kappa \nabla u\cdot \nabla v\ dx.$

Since the permeability $\kappa$ can be large in the high contrast case, we will use implicit time discretization scheme to ensure the stability of the discrete system. In this work, we consider the implicit Euler scheme, and the proposed method can be extended to other time discretization schemes. The fully discretized scheme is to find $u^{n+1}\in V_{h,0}$
such that 
\[
(u^{n+1},v)+\Delta t a(u^{n+1},v)=(u^{n},v)+\Delta t 
 (f,v), \quad\forall v\in V_{h,0}.
\] 

Denote by $A$ and $M$ the fine grid stiffness and mass matrix, respectively. In a matrix form, we need to solve 
$$(\Delta t A+M)U^{n+1}=MU^n+F,$$
where $U^n$ is the coefficient vector of $u^n \in V_h$. 

It is in general challenging to solve such system with matrix $\Delta t A+M$ since the condition number of $A$ is proportional to $\eta$ and $h^{-2}$, where $h$ refers to the fine mesh size and $\eta$ is the contrast ratio. Using an iterative method to solve the large-scale system of equations can be highly efficient, as it typically requires only matrix--vector multiplications and avoids the prohibitive computational cost and memory consumption associated with direct factorization. However, the convergence rate of such methods is often strongly influenced by the spectral properties of the system matrix. Preconditioning therefore serves as an essential technique to control the condition number and improve the spectral distribution, thereby accelerating convergence and ensuring robustness. In this work, we aim to construct preconditioners that allow us to balance computational efficiency with implementation complexity for this time-dependent high-contrast problem.

\subsection{Two-level overlapping preconditioner}
To tackle the aforementioned problem, we focus on the \textit{two-level overlapping domain decomposition methods} and utilize local information to construct a minimal dimensional coarse space to obtain an overlapping domain decomposition preconditioners such that the condition number of the perconditioned system is \textit{contrast-independent}. 


 Let $\mathcal{T}^H$ be a partition of $\Omega$ with mesh size $H$, and $K_i (i=1,\cdots, N)$ are coarse elements.  We consider $\Omega=\cup_{i}\omega_{i}$ where $\omega_{i}$ are overlapping local domains by extending layers of fine elements to each coarse element $K_i$, and $V_{i}=V_{h,0}(\omega_{i})$. Therefore, in the additive two-level overlapping Schwarz method \cite{toselli2004domain}, the finite element space is decomposed as 
 $$V_{h,0}= R_0^T V_0+\sum_{i=1}^{N} R_i^TV_i,$$
 where $V_0$ is a global coarse space defined on $\mathcal{T}^H$, and $R_i^T: V_i \to V_{h,0}, i=0,\cdots, N$ are some interpolation operators. The decomposition of the space plays an essential role in the development and analysis of the method.

Once the coarse space $V_0$ is given, we can obtain the matrix of the global solver $\tilde{A}_0 = R_0 (\Delta tA+M) R_0^T$, where $A$ and $M$ are the fine scale stiffness and mass matrix. Denote by $\tilde{A}_i=(\Delta tA_i+M_i)$ the matrix of local solvers, and $\tilde{A}_{\text{pre}}^{-1}$ the preconditioner to the system. We note that the coarse space $V_0$, the local spaces $V_i$ and the corresponding basis matrices can be constructed in an offline manner. Then $\tilde{A}_{\text{pre}}^{-1}$ has the form 
\[
\tilde{A}_{\text{pre}}^{-1}= R_0^T \tilde{A}_0^{(-1)} R_0  + \sum_{i=1}^N
R_i^T \tilde{A}_i^{(-1)} R_i.  \] 
With this preconditioner, one can iteratively solve the global system in an efficient way, $$\tilde{A}_{\text{pre}}^{-1} (\Delta tA+M) U^{n+1} = \tilde{A}_{\text{pre}}^{-1}  (MU^n+f).$$
 
 The effectiveness of this approach hinges on constructing a preconditioner such that the condition number of $\tilde{A}_{\text{pre}}^{-1} (\Delta tA+M)$ is small and independent of physical parameters $\eta$.

For high-contrast multiscale diffusion problems, classical coarse spaces, such as the piecewise linear finite element space $V_0 = P_1(\mathcal{T}^H)$ on a coarse grid, often fail to capture the essential variability of the permeability coefficient $\kappa$. When $\kappa$ contains high-conductivity channels, the standard coarse interpolation cannot adequately represent the multiscale features. Consequently, the condition number of the resulting preconditioned system typically scales with the contrast $\eta$, rendering classical iterative solvers inefficient.

To address this fundamental challenge, various advanced methodologies have been developed to construct more robust coarse spaces. These broadly include domain decomposition type methods and multiscale/multigrid type approaches. A class of these methods constructs the coarse space $V_0$ based on the generalized multiscale finite element method (GMsFEM). In this approach, local eigenvalue problems formulated via Rayleigh quotients over snapshot spaces are solved to extract the important modes, particularly those associated with high-conductivity channels, yielding preconditioners whose condition numbers are independent of the contrast.
In another work, the constraint energy minimizing generalized multiscale finite element method (CEM-GMsFEM) constructs the coarse space by solving a constrained energy minimization problem \cite{GalvisChungEfendievLeung2018}. By enforcing orthogonality to an appropriate auxiliary space and utilizing oversampling techniques with sufficiently large overlap, this approach can produce preconditioner with a condition number close to $1$. However, the construction requires solving local problems on enlarged oversampled regions, resulting in a higher computational cost.

 Based on the constraint energy minimizing (CEM) and nonlocal multicontinnum (NLMC) method, in this paper, we will introduce an iterative approach to improve the efficiency of constructing robust coarse space $V_0$ at the offline stage. To further improve efficiency of the proposed preconditioner, we will show that the low-dimensional dominant component of the coarse space can be used to reduce computational complexity. The details will be shown in latter sections.

\section{The construction of coarse space}  
\subsection{Relaxed NLMC Space}\label{sec:relaxed_NLCM}

In this section, we discuss the construction of the Nonlocal Multiscale Continuum (NLMC) space $V_H$ by solving a relaxed constraint energy minimization problem. Motivated by the CEM-GMsFEM, the standard construction of NLMC space enforces the multiscale basis functions to be strictly orthogonal to a local auxiliary space $V_{\text{aux}}$ spanned by eigenfunctions corresponding to small eigenvalues. This guarantees a mesh-dependent convergence that is independent of the contrast when using sufficiently large oversampling domains. Instead of solving the constrained energy minimization problem, one can adopt the relaxed formulation that replaces the hard orthogonality constraints with a penalized energy functional. 
 
Specifically, we first define the auxilliary space. For each coarse element $K$, let $K_{h}^{i}\subset K$ be the $i$-th non-overlapping open connected lipschitz domain and $K_{l}=K\backslash(\cup_{i}K_{h}^{i})$. Let $\eta_{0},\kappa_{0}$ be constants such that 
\begin{align*}
    \cfrac{\max_{x\in K_{h}^{i}}\kappa(x)}{\min_{x\in K_{h}^{i}}\kappa(x)}\leq\eta_{0},\quad 
    \min_{x\in K_{h}^{i}}\kappa(x)\geq\kappa_{0},\quad 
    \max_{x\in K_{l}}\kappa(x)<\kappa_{0},
\end{align*}
and there exists a constant $C$ independent of $H$ and $K$ such that 
\[
\int_{K_{h}^{i}}|u|^{2}\leq CH^2\int_{K_{h}^{i}}|\nabla u|^{2}
\quad\text{for all }u\text{ with }\int_{K_{h}^{i}}u=0.
\]

For later analysis, define $\|u\|_a^2:=\int_{\Omega}\kappa|\nabla u|^2$, and $s(\,u,v\,):=H^{-2}\int_{\Omega} \kappa uv$, $\|u\|_s^2:= \int_{\Omega} \kappa u^2$. The corresponding local forms are defined by restricting the same integral to the indicated subdomain.

For each $K\in\mathcal{T}_{H}$, we consider $V_{\text{aux}}^K=V_{\text{aux},1}^K+V_{\text{aux},2}^K$ where $V_{\text{aux},1}^K=\text{span}\{\chi_{K_{h}^{i}}|K_h^i \subset K\}$,
$V_{\text{aux},2}^K=\text{span}\{\chi_{K_{l}}|K_{l}=K\backslash(\cup_{i}K_{h}^{i})\}$, where $\chi$ is the characteristic function. Then the global auxillary space is $V_{\text{aux}}:=\oplus_{K\in\mathcal{T}_{H}} V_{\text{aux}}^K$.

Let $\pi_{K}:H^{1}(K)\to L^{2}(K)$ be the projection operator, satisfying
\begin{equation} \label{eq:poincare_1}
    \|(I-\pi_{K})u\|_{s(K)}^{2}\leq C_{\pi}\|u\|_{a (K)}^{2}
\end{equation}
for all $u\in H^1(K)$, where $C_\pi$ is independent of $H$ and of the contrast. The global projection operator is $\pi: H^1 (\Omega) \to L^2(\Omega)$ with $\pi(u):=\sum_{K}\pi_{K}(u|_{K})$. 

Next, we will solve an unconstrained minimization problem to obatain the relaxed version of NLMC basis. For a given basis $\psi_{j,K}\in V_{\text{aux},j}, \; j=1,2$, the corresponding NLMC basis function $\phi_{j,K}$ can be obtained by solving 
\begin{equation*}
    \phi_{j,K} = \operatorname*{arg\,min}_{\phi \in V_{h,0}} 
\left\{ a(\phi,\phi) + s\big(\pi(\phi) - \psi_{j,K},\, 
\pi(\phi) - \psi_j^{(i)}\big) \right\}.
\end{equation*}
The penalty term $s\big(\pi(\phi) - \psi_{j,K},\, 
\pi(\phi) - \psi_j^{(i)}\big)$ enforces approximate orthogonality to the auxiliary modes rather than exact orthogonality, thereby allowing the basis functions to be computed on smaller oversampled regions while retain the approximation power \cite{ChungEfendievLeung2018CEM}. 
Equivalently, the minimizer satisfies the variational equation
\begin{equation}
\label{eq:NLMC_relax}
    a_(\phi_{j,K},v)+s(\pi(\phi_{j,K}),\pi(v))=s(\psi_{j,K},\pi(v)),\quad\forall v\in V_{h,0}.
\end{equation}

Then the NLMC space can be defined as $V_{H}:=\text{Span}_{j,K}\{\phi_{j,K}\}$, and the coarse space $V_{H}$ can be natually decomposed into $V_{H,1}$ and $V_{H,2}$. We will show that the subspace $V_{H,1}$ can capture the ``fast" or the dominant component of the solution, therefore serving as an important role in our construction of the preconditioner.

\subsection{Iterative construction of $V_{H}$}\label{sec:iterative_construction}
In this section, we will discuss an iterative construction of the coarse space $V_H$. Let $v_{h,j}$ be the nodal basis function of the finite element space $V_{h,0}$, then the global stiffness matrix $\textbf{A}$ and projected s-norm matrix $\textbf{S}$ are defined as 
\[
A_{ji}=a(v_{h,i},v_{h,j}),S_{ji}=s(\pi v_{h,i},\pi v_{h,j}).
\]

In this section, we let
\[
V_{h,0}=\sum_{i} R_i^TV_{i}.
\]
Then we denote by $\textbf{A}_{i}$ and $\textbf{S}_i$ the corresponding stiffness matrix and projected s-norm matrix of $V_i$. 

Let $\{\chi_i\}_i$ be a partition of unity satisfying the following conditions: $\chi_i\in V_h$ is supported in $\omega_i$ with $\sum_i\chi_i =1$, $|\nabla \chi_i|\leq CH^{-1}$ and $\sum_i\chi^2_i\leq 1$. Let $\phi^{\text{glo}}_{j,K}$ be the global basis function obtained from the relaxed NLMC approach, that is,
$$(\textbf{A}+\textbf{S})\phi^{\text{glo}}_{j,K} =\tilde{f}_{j}.$$
where $(\tilde{f}_{j})_{k}=s(\pi\psi_{j,K},\pi v_{k})$.

Define $\textbf{A}^{-1}_{S,loc}=\sum_{i=1}R^T_{i}(\textbf{A}_{i}+\textbf{S}_{i})^{-1}R_{i}$,
we can prove that the conditional number of $\textbf{A}^{-1}_{S,loc}(\textbf{A}+\textbf{S})$ is in the following theorem. 
\begin{thm} Assume that $V_{\text{aux}}$ satisfies \eqref{eq:poincare_1}. Then
\[
\text{cond}\Big(\textbf{A}^{-1}_{S,loc}(\textbf{A}+\textbf{S})\Big)=O(1)
\]
which is independent of contrast, the fine mesh size $h$ and the coarse mesh size $H$.

\end{thm}
\begin{proof}
Summing \eqref{eq:poincare_1} over all coarse elements, we can have
\begin{equation}
    \|(I-\pi)u\|_{s}^{2}
    \leq
    C_{\pi}\|u\|_{a}^{2}.
    \label{eq:global-poincare}
\end{equation}
Define $b(u,v)=a(u,v)+s(\pi u,\pi v)$ and
$\|u\|_{b}^{2}:=b(u,u)
    =\|u\|_{a}^{2}+\|\pi u\|_{s}^{2}$. Then
\begin{equation}    \label{eq:mass-controlled-by-b}
\begin{aligned}
    \|u\|_{s}^{2}
    &\leq
    2\|\pi u\|_{s}^{2}
    +
    2\|(I-\pi)u\|_{s}^{2} \leq
    2\|\pi u\|_{s}^{2}
    +
    2C_\pi \|u\|_{a}^{2} \leq C\|u\|_{b}^{2}.
\end{aligned}    
\end{equation}

The constant in \eqref{eq:mass-controlled-by-b} is independent of $H$ and of the contrast. 

%
For any $u\in V_h$, define $u_i:=I_h(\chi_i u)\in V_i$, where $I_h$ is the nodal interpolant satisfying $\|I_h(v)\|_a\leq C_I\|v\|_a$ and $\|I_h(v)\|_s\leq C_I\|v\|_s$ for all $v\in C^0 \cap P_2(\mathcal T^h)$, and $\chi_i$ is the partition of unity. Since $I_h$ is linear and is the identity on $V_h$, $\sum_i u_i=I_h(\sum_i \chi_i u ) = I_h(u)=u$. Note that $\nabla(\chi_i u) = \chi_i\nabla u+u\nabla\chi_i$; by the stability of $I_h$ and \eqref{eq:mass-controlled-by-b}, we have
\begin{equation}\label{eq:mass-controlled}
   \begin{aligned}
  &  \sum_i \|I_h(\chi_i u)\|_{a(\omega_i)}^{2} \leq C_I^2\sum_i \|\chi_i u\|_{a(\omega_i)}^{2}\\
    &\leq
    C\sum_i\int_{\omega_i}\kappa \chi_i^2|\nabla u|^2\,dx
    +
    C\sum_i\int_{\omega_i}\kappa u^2|\nabla\chi_i|^2\,dx \\
    &\leq
    C\|u\|_{a}^{2}
    +
    C H^{-2}\int_{\Omega}\kappa u^2\,dx \leq C\|u\|_{b}^{2}
\end{aligned} 
\end{equation}

Moreover, using the property of $\pi$, we have
\begin{equation}\label{eq:stable-decomp-s}
\begin{aligned}
    \sum_i \|\pi(I_h(\chi_i u))\|_{s(\omega_i)}^{2}
    &\leq
    C\sum_i \|I_h(\chi_i u)\|_{s(\omega_i)}^{2} \leq
    C\|u\|_{s}^{2}\leq C\|u\|_{b}^{2}.
\end{aligned}
\end{equation}

Combining \eqref{eq:mass-controlled} and \eqref{eq:stable-decomp-s} yields the stable decomposition
\begin{equation}
    \sum_i b(u_i,u_i)
    \leq
    C_{\mathrm{sd}}\,b(u,u).
    \label{eq:stable-decomposition}
\end{equation}

On the other hand, since $u=\sum_i I_h(\chi_i u)=\sum_i u_i$ and the supports of the $u_i$ have uniformly bounded overlap,
\begin{equation}
   \left\| u\right\|_{a}^{2} = \left\|\sum_i I_h(\chi_i u)\right\|_{a}^{2}
    \leq
    C_{\mathrm{ov}}
    \sum_i \| I_h(\chi_i u)\|_{a}^{2}
    =C_{\mathrm{ov}}\sum_i\|u_i\|_a^2.
    \label{eq:bounded-overlap}
\end{equation}
Similarly, 
\begin{equation}
   \left\| \pi u\right\|_{s}^{2} \leq \left\| u\right\|_{s}^{2}= \left\|\sum_i I_h(\chi_i u)\right\|_{s}^{2}
    \leq
    C_{\mathrm{ov}}
    \sum_i \| I_h(\chi_i u)\|_{s}^{2} \leq  C_{\mathrm{ov}}
    \sum_i \| u_i\|_{b}^{2}.
    \label{eq:bounded-overlap2}
\end{equation}
Combining \eqref{eq:bounded-overlap} and \eqref{eq:bounded-overlap2}, we have 
\begin{equation}
b(u,u)  
    \leq
    C_{\mathrm{ov}}\,\sum_i b(u_i,u_i).
\end{equation}

Consequently,
\begin{equation*}
\begin{aligned}
    \text{cond}\Big(\textbf{A}^{-1}_{S,loc}(\textbf{A}+\textbf{S})\Big)
    &=
    \frac{
    \lambda_{\max}\Big(\textbf{A}^{-1}_{S,loc}(\textbf{A}+\textbf{S})\Big)
    }{
    \lambda_{\min}\Big(\textbf{A}^{-1}_{S,loc}(\textbf{A}+\textbf{S})\Big)
    } \leq
    C_{\mathrm{sd}}C_{\mathrm{ov}}.
\end{aligned}
\end{equation*}
Since $C_{\mathrm{sd}}$ and $C_{\mathrm{ov}}$ are independent of the contrast, the fine mesh size $h$ and of the coarse mesh size $H$. This completes the proof.

\end{proof}

With the above result, one can use the preconditioned conjugate gradient (PCG) method to solve $\phi_{j,K}$ iteratively. Let $\widehat\phi_{j,K}^{(m)}$ be the functions obtained after $m$ iterations of the basis solver and set $\widehat V_{H}^{(m)}:=\operatorname{range}(\widehat\phi_{j,K}^{(m)}), \;j=1,2$. Since $\tilde{A}=\Delta tA+M$ is symmetric positive definite, it induces the norm
\[
 \|v\|_{\tilde a}^{2}:=\tilde a(v,v)=\Delta t\,a(v,v)+(v,v).
\]
We further define $\tnorm {v}^{2}
 :=\|v\|_{\widetilde a}^2+\Delta t \|v\|_s^2$, and let $\mathcal P_m^{\mathrm{loc}}:V_{h,0}\to\widehat V_{H,1}^{(m)}$ be the
orthogonal projection in the inner product induced by
$\tnorm{\cdot}$, and set $T_m:=\mathcal P_m|_{V_{H,1}}$. We make the following assumption.

\begin{assumption}\label{ass:inexact_basis}
The coarse space $V_{H}$ and computed coarse space $\widehat V_{H}^{(m)}$ have the same dimension, and there is a number $0\leq\varepsilon_m\leq\varepsilon_*<1$ such that
\begin{equation}
 \label{eq:inexact-basis-error}
 \tnorm{v_0-T_mv_0}
 \leq \varepsilon_m\|v_0\|_{\widetilde a},
 \qquad v_0\in V_{H,1}.
\end{equation}
The threshold $\varepsilon_*$ is independent of $h$, $H$, and the coefficient contrast.
\end{assumption}
Due to the exponential convergence property, we note that when $m$ is in order of $\log(\eta)$, the assumption can be satisfied \cite{ChungEfendievLeung2018CEM}.

\section{Fast scale preconditioner}\label{sec:preconditioner}
In this section, we discuss the construction of efficient preconditioners based on the coarse space defined in the previous section. 

Let $\pi_j$ denote the $s$-orthogonal projection onto
$V_{\mathrm{aux},j}$, so that $\pi=\pi_1+\pi_2$. We use the NLMC
projection $P_1:V_{h,0}\to V_{H,1}$ characterized by
\begin{align}
&s(\pi_1P_1u,\pi_1v)=s(\pi_1u,\pi_1v),
v\in V_{H,1}, \label{eq:P1-moment-projection}
\end{align}
Therefore, we have 
\begin{align}
a(P_{1}u,P_{1}u)+s(\pi P_{1}(u),\pi P_{1}u)=a(P_{1}u,u)+s(\pi P_{1}(u),\pi u).
\label{eq:P1-energy-identity}
\end{align}
Since $\pi_1(V_{H,1})=V_{\mathrm{aux},1}$, \eqref{eq:P1-moment-projection} implies
\begin{equation}\label{eq:P1-moments}
 \int_{K_h^j}(u-P_1u)\,dx=0
 \qquad\text{for every high-conductivity component }K_h^j.
\end{equation}

In this work, we will propose two different ways to construct coarse spaces.
\begin{itemize}
\item \textbf{Case 1}: For $\Delta t\sim \mathcal{O} (H^{2})$, we can consider the coarse space $V_{0}$ to be simply $V_{H,1}$.
\item \textbf{Case 2}: In a general case, we can consider the coarse space $V_{0}$ to be
$V_{H,1}+W_{H}$ where $W_{H}$ can be standard coarse grid finite element
space.
\end{itemize}

We remark that for the first case, given a small step size $\Delta t$, it is enough to use a preconditioner which only handles the high permeability part. For the domain with small high parameter region, this method can significantly reduce the space dimension of the coarse space. On the other hand, recall that the partially explicit scheme\cite{splitting1, splitting2, splitting-multirate} was previously proposed to handle multiscale parabolic problems. In that approach, one first decomposes the multiscale space $V_{H}$ into $V_{H,1}$ and $V_{H,2}$, then use implicit scheme for the part of solution in $V_{H,1}$ and explicit scheme the part of solution in $V_{H,2}$. It was proved that the partially explicit scheme is stable given time step size in $O(H^{2})$ and is independent of contrast \cite{splitting1, splitting2}. In Case 1, $V_{0}$ is taken as $V_{H,1}$. Thus the coarse scale matrix $\tilde{A}_0 = \Delta t A_0+ M_0$ only corresponds to the fast flow part. Consequently, Case 1 is can be viewed as using a coarse grid partially explicit scheme as a preconditioner to the fine grid fully implicit scheme.

For the second case, we take $V_0 = V_{H,1}+W_H$. Here, $W_H$ can be the finite element space or multiscale finite element space on the coarse grid. Then we have an enriched coarse space and can take larger time step sizes. Since $W_H$ is easy to obtain, the computational cost can be substantially reduced.

Once the coarse space $V_0$ is constructed, the preconditioner $\tilde{A}_{\text{pre}}$ is then defined as $\tilde{A}_{\text{pre}}=(R_{0}^{T}(\Delta tA_0+M_0)^{-1}R_{0} + \sum_{i=1}^N R_{i}^{T}(\Delta tA_{i}+M_{i})^{-1}R_{i})^{-1}$. Here, $A_0$ and $M_0$ are coarse scale stiffness and mass matrices, $A_{i}$ and $M_{i}$ ($i=1,\cdots, N$) are the stiffness matrix and mass matrix corresponding to the local finite element space $V_{i}$ ($i=1,\cdots, N$).

\subsection{The condition number for Case 1}
 In the following, we will show that the condition number of proposed preconditioner $\tilde{A}_{\text{pre}}$ estimation for \textbf{Case 1}. We start with two useful lemmas.

\begin{lem}\label{lem1}
Consider the coarse space $V_{0}=V_{H,1}$. We have 
\[
H^{-2}\int_{\Omega}\kappa\Big|\Big((I-P_{1})(u)\Big)\Big|^{2}\lesssim\eta_{0}\|u-P_{1}(u)\|_{a}^{2}+H^{-2}\kappa_{0}\|u-P_{1}(u)\|^{2}
\]
 for any $u\in V_{h}$. Here $P_{1}(u)\in V_{H,1}$ satisfies $\int_{K_{h}^{i}}u=\int_{K_{h}^{i}}P_{1}(u)$
for all $K_{h}^{i}$.
\end{lem}

\begin{proof}
Put $w=u-P_1u$, then on every high-conductivity component,
$\int_{K_h^i}w=0$ by \eqref{eq:P1-moments}. The scaled Poincar\'e inequality and the bounded
variation of $\kappa$ within $K_h^i$ therefore give
\[
 H^{-2}\int_{K_h^i}\kappa |w|^2
 \leq C\eta_0\int_{K_h^i}\kappa|\nabla w|^2.
\]
On the low-conductivity part, $\kappa\leq\kappa_0$, hence
\[
 H^{-2}\int_{K_l}\kappa |w|^2
 \leq H^{-2}\kappa_0\|w\|_{L^2(K_l)}^2.
\]
Summing both estimates over the coarse elements $K$ proves
\[
 H^{-2}\int_\Omega\kappa|w|^2
 \lesssim \eta_0\|w\|_a^2+H^{-2}\kappa_0\|w\|^2.
\]
\end{proof}

\begin{lem}\label{lem2}
Consider the coarse space $V_{0}=V_{H,1}$. We have 
\[
\sum_{i}\|I_h(\chi_{i}(u-P_{1}(u)))\|_{a}^{2}\lesssim (1+\eta_{0}) \|u-P_{1}(u)\|_{a}^{2}+H^{-2}\kappa_{0}\|u-P_{1}(u)\|^{2}
\]
and 
\[
\sum_{i}\|I_h(\chi_{i}(u-P_{1}(u)))\|^{2}\lesssim \|u-P_{1}(u)\|^{2}
\]
 for any $u\in V_{h}$.
\end{lem}

\begin{proof}
Put $w=u-P_1u$. By the $H^1$ stability of the nodal interpolant, the
product rule, $\sum_i\chi_i^2\leq1$, and
$|\nabla\chi_i|^2\lesssim H^{-2}$, we have
\begin{align*}
 \sum_i\|I_h(\chi_iw)\|_a^2
 &\lesssim\sum_i\|\chi_iw\|_a^2\lesssim H^{-2}\int_\Omega\kappa|w|^2+\|w\|_a^2.
\end{align*}
Using Lemma \ref{lem1}, we have 
\begin{align*}
\sum_i\|I_h(\chi_iw)\|_a^2
&\lesssim(1+\eta_{0})\|w\|_{a}^{2}+H^{-2}\kappa_{0}\|w\|^{2}.
\end{align*}

The $L^2$ stability of $I_h$ and the bounded overlap give
\begin{align*}
\sum_i\|I_h(\chi_iw)\|^2
&\lesssim\sum_i\|\chi_iw\|^2
=\int_\Omega\sum_i\chi_i^2|w|^2
\leq\|w\|^2.
\end{align*}
\end{proof}

\begin{lem}\label{lem:P1-transient-stability}
If $c_tH^2\leq\Delta t\leq C_tH^2$, then the NLMC projection satisfies
\[
 \|P_1u\|_{\widetilde a}\leq C_P\|u\|_{\widetilde a},
\]
with a constant independent of the contrast.
\end{lem}
\begin{proof}

Taking $v=P_1u$ in \eqref{eq:P1-moment-projection} cancels the
$\pi_1$ terms in \eqref{eq:P1-energy-identity}. Consequently,
\begin{align}
 &a(P_1u,P_1u)+\|\pi_2P_1u\|_s^2
 =a(P_1u,u)+s(\pi_2P_1u,\pi_2u)\notag\\
 &\leq \frac12\left(a(P_1u,P_1u)+\|\pi_2P_1u\|_s^2\right)
 +\frac12\left(a(u,u)+\|\pi_2u\|_s^2\right).
 \label{eq:P1-a-s-stability}
\end{align}
Thus
\begin{equation}\label{eq:P1-a-bound}
 a(P_1u,P_1u)+\|\pi_2P_1u\|_s^2
 \leq a(u,u)+\|\pi_2u\|_s^2
 \leq a(u,u)+\kappa_0 H^{-2}\|u\|^2.
\end{equation}

It remains to estimate the unweighted mass norm. By
\eqref{eq:global-poincare} and the decomposition $\pi=\pi_1+\pi_2$,
\begin{align*}
 \|P_1u\|^2
 &\lesssim H^2\left(a(P_1u,P_1u)+\|\pi_2P_1u\|_s^2\right)
 +\|\pi_1P_1u\|^2.
\end{align*}
Identity \eqref{eq:P1-moment-projection} and
$\pi_1(V_{H,1})=V_{\mathrm{aux},1}$ imply
$\pi_1P_1u=\pi_1u$. 
Using \eqref{eq:P1-a-bound}, we obtain
\begin{equation}\label{eq:P1-L2-bound}
 \|P_1u\|^2
 \lesssim H^2a(u,u)+(1+\kappa_0)\|u\|^2.
\end{equation}
Multiplying \eqref{eq:P1-a-bound} by $\Delta t$ and combining it with
\eqref{eq:P1-L2-bound}, with the condition $c_tH^2\leq\Delta t\leq C_tH^2$, we get the desired result.
\end{proof}

Finally, we have the following theorem.
\begin{thm}\label{thm:case1}
Consider the coarse space $V_{0}=V_{H,1}$. If
$c_tH^2\leq\Delta t\leq C_tH^2$ and $\eta_0$ and $\kappa_0$ are
uniformly bounded, then
we have $\text{cond}\Big((R_{0}^{T}\tilde{A}_{0}^{-1}R_{0}+\sum_{i}R_{i}^{T}\tilde{A}_{i}^{-1}R_{i})\tilde{A}\Big)=O(1)$
with $\tilde{A}=\Delta tA+M$, $\tilde{A}_{i}=\Delta tA_{i}+M_{i}$.
\end{thm}

\begin{proof}

Set $u_0:=P_1u$. For $u\in V_h$, set $u_i:=I_h(\chi_i(u-u_0))\in V_i$. Since $I_h$ is linear and is the identity on $V_h$,
\[
 u_0+\sum_i u_i
 =u_0+I_h\!\left(\sum_i\chi_i(u-u_0)\right)
 =u_0+I_h(u-u_0)=u.
\]
The estimates of Lemma~\ref{lem2}, together with
$\Delta t\,\kappa_0\lesssim H^2$, give a constant
$C_{\rm loc}$ independent of $h$, $H$, and the contrast such that
\begin{equation}\label{eq:case1-local-stability}
 \sum_i\|u_i\|_{\tilde a}^{2}
 \leq C_{\rm loc}\|u-u_0\|_{\tilde a}^{2}.
\end{equation}
The bounded overlap gives a constant $C_{\rm ov}$, also independent of $h$, $H$, and the contrast, such that
\begin{equation}\label{eq:case1-overlap}
 \|u-u_0\|_{\tilde a}^{2}
 =\Big\|\sum_i u_i\Big\|_{\tilde a}^{2}
 \leq C_{\rm ov}\sum_i\|u_i\|_{\tilde a}^{2}.
\end{equation}

Moreover, by Lemma~\ref{lem:P1-transient-stability}, we have the control on the coarse component:
\begin{equation}\label{eq:coarse-stability}
 \|u_0\|_{\tilde a}\leq C_P\|u\|_{\tilde a}.
\end{equation}
Using \eqref{eq:case1-local-stability} and
$\|u-u_0\|_{\tilde a}\leq\|u\|_{\tilde a}+\|u_0\|_{\tilde a}$, we obtain the stable decomposition
\begin{equation}\label{eq:case1-stable-decomposition}
\begin{aligned}
 \|u_0\|_{\tilde a}^{2}+\sum_i\|u_i\|_{\tilde a}^{2}
 &\leq C_P^2\|u\|_{\tilde a}^{2}+C_{\rm loc}\|u-u_0\|_{\tilde a}^{2}\\
 &\leq \bigl(C_P^2+2C_{\rm loc}(1+C_P^2)\bigr)
 \|u\|_{\tilde a}^{2}.
\end{aligned}
\end{equation}
Conversely, for arbitrary $v_0\in V_0$ and $v_i\in V_i$, bounded
overlap gives
\begin{equation}\label{eq:case1-continuity}
 \left\|v_0+\sum_i v_i\right\|_{\tilde a}^{2}
 \leq 2\|v_0\|_{\tilde a}^{2}
 +2C_{\rm ov}\sum_i\|v_i\|_{\tilde a}^{2}.
\end{equation}

Therefore, we have 
\[
 \operatorname{cond}\Bigl(\bigl(R_{0}^{T}\tilde A_{0}^{-1}R_{0}
 +\sum_iR_i^{T}\tilde A_i^{-1}R_i\bigr)\tilde A\Bigr)
 =O(1),
\]
where the constant is independent of $h$, $H$, and the contrast.
\end{proof}

In the following, we prove the result when the basis functions in $V_{H,1}$ is computed by the iterative method.
\begin{thm}
\label{thm:case1-inexact}
Assume the hypotheses of Theorem~\ref{thm:case1} and
Assumption~\ref{ass:inexact_basis}, with
$\varepsilon_m\leq\varepsilon_*<1$. Let $\widehat R_0$ be the restriction
matrix associated with $\widehat V_{H,1}^{(m)}$ and let
\[
 \widehat{\widetilde A}_0
 :=\widehat R_0\widetilde A\widehat R_0^T.
\]
Then the two-level additive Schwarz preconditioner formed with the computed
coarse space satisfies
\begin{equation}
\label{eq:inexact-condition-number}
 \operatorname{cond}\!\left(
 \left(\widehat R_0^T\widehat{\widetilde A}_0^{-1}\widehat R_0
 +\sum_iR_i^T\widetilde A_i^{-1}R_i\right)\widetilde A\right)
 \leq C(\varepsilon_*)=O(1),
\end{equation}
where $C(\varepsilon_*)$ is independent of $h$, $H$, and the
coefficient contrast.
\end{thm}

\begin{proof}
Let $u_0=P_1u$ be the ideal coarse component used in the proof of
Theorem~\ref{thm:case1}, and define
\[
 \widehat u_0:=T_mu_0,\qquad
 w:=u-u_0,\qquad e:=u_0-\widehat u_0.
\]
By Assumption~\ref{ass:inexact_basis} and
Lemma~\ref{lem:P1-transient-stability},
\begin{align}
 \tnorm{e}
 &\leq \varepsilon_m\|u_0\|_{\widetilde a}
 \leq \varepsilon_m C_P\|u\|_{\widetilde a},
 \label{eq:inexact-e-bound}\\
 \|\widehat u_0\|_{\widetilde a}
 &\leq (1+\varepsilon_m)C_P\|u\|_{\widetilde a}.
 \label{eq:inexact-coarse-bound}
\end{align}

For every $z\in V_h$,
\begin{equation}
\label{eq:localization-product-bound}
 \sum_i\|I_h(\chi_i z)\|_{\widetilde a}^{2}
 \leq C
 \left(\|z\|_{\widetilde a}^{2}
 +\Delta t H^{-2}\int_\Omega\kappa|z|^2\right)
 =C\tnorm{z}^{2},
\end{equation}
where $C$ is independent of $h$, $H$, and the contrast.
Define
\[
 \widehat u_i:=I_h\bigl(\chi_i(u-\widehat u_0)\bigr)
 =I_h(\chi_iw)+I_h(\chi_i e).
\]
The nodal partition of unity yields
$u=\widehat u_0+\sum_i\widehat u_i$. Using
\eqref{eq:case1-local-stability},
\eqref{eq:localization-product-bound}, and
\eqref{eq:inexact-e-bound}, we obtain
\begin{align*}
 \sum_i\|\widehat u_i\|_{\widetilde a}^{2}
 &\leq 2C_{\mathrm{loc}}\|w\|_{\widetilde a}^{2}
   +2C\tnorm{e}^{2}\\
 &\leq
 \left(2C_{\mathrm{loc}}(1+C_P)^2
 +2C\varepsilon_m^2C_P^2\right)
 \|u\|_{\widetilde a}^{2}.
\end{align*}
Together with \eqref{eq:inexact-coarse-bound}, this proves a stable
decomposition in the computed spaces with constant bounded by $(1+\varepsilon_m)^2C_P^2
 +2C_{\mathrm{loc}}(1+C_P)^2
 +2C\varepsilon_m^2C_P^2$.
This constant is uniform when $\varepsilon_m\leq\varepsilon_*$. The reverse continuity estimate is unchanged because it depends only on bounded overlap. This completes the proof.
\end{proof}

\subsection{The condition number for Case 2}

Next, we give the result for the condition number of proposed preconditioner $\tilde{A}_{\text{pre}}$ in \textbf{Case 2}.
\begin{thm} \label{thm:case2}
Let $V_0=V_{H,1}+W_H$, and suppose that there is a linear operator
$I_0:V_{h,0}\to V_0$ for which
\begin{align}
 \|I_0u\|_{\widetilde a}&\leq C_0\|u\|_{\widetilde a},
 \label{eq:case2-coarse-stability}\\
 H^{-2}\int_\Omega\kappa|u-I_0u|^2
 &\leq C_A\,a(u-I_0u,u-I_0u).
 \label{eq:case2-weighted-approximation}
\end{align}
Assume that $C_0$ and $C_A$ and the overlap multiplicity are independent
of $h$, $H$, and the contrast. Then, for every $\Delta t>0$,
\[
\operatorname{cond}\Bigl(\bigl(R_0^T\widetilde A_0^{-1}R_0
+\sum_iR_i^T\widetilde A_i^{-1}R_i\bigr)\widetilde A\Bigr)=O(1),
\]
where $\widetilde A=M+\Delta tA$ and
$\widetilde A_i=M_i+\Delta tA_i$.
\end{thm}

\begin{proof}
Put $u_0=I_0u$, $w=u-u_0$, and
$u_i=I_h(\chi_iw)$. Since $I_h$ is linear and is the identity on $V_h$,
$u=u_0+\sum_i u_i$. The stability of $I_h$, bounded overlap,
the product rule, and \eqref{eq:case2-weighted-approximation} imply
\begin{align*}
 \sum_i\|u_i\|_{\widetilde a}^2
 &\lesssim \|w\|^2+\Delta t\left(a(w,w)
       +H^{-2}\int_\Omega\kappa|w|^2\right)\\
 &\lesssim \|w\|^2+\Delta t\,a(w,w)
 =\|w\|_{\widetilde a}^2.
\end{align*}
By \eqref{eq:case2-coarse-stability},
$\|w\|_{\widetilde a}\leq(1+C_0)\|u\|_{\widetilde a}$ and
$\|u_0\|_{\widetilde a}\leq C_0\|u\|_{\widetilde a}$. Hence
\[
 \|u_0\|_{\widetilde a}^2+
 \sum_i\|u_i\|_{\widetilde a}^2
 \lesssim\|u\|_{\widetilde a}^2.
\]
The reverse continuity estimate follows from bounded overlap. The abstract
additive Schwarz lemma now gives the asserted condition-number bound.
\end{proof}

We note that Theorem~\ref{thm:case2} is stated based on verifiable
interpolation properties. In practice, we can take standard finite element or MsFEM space as $W_H$. The further details are omitted here.

\section{Numerical Experiments}
\label{sec:numerical}

\subsection{A time independent permeability field}
In this section, we present numerical results to demonstrate the performance of the proposed method. We consider the computational domain $\Omega = [0,1]^2$ with a time-independent high-contrast permeability field $\kappa$ and piecewise constant source term,shown in Figure~\ref{fig:ex1_a_f}. 

\begin{figure}[ht!]
\centering
\includegraphics[scale=0.35]{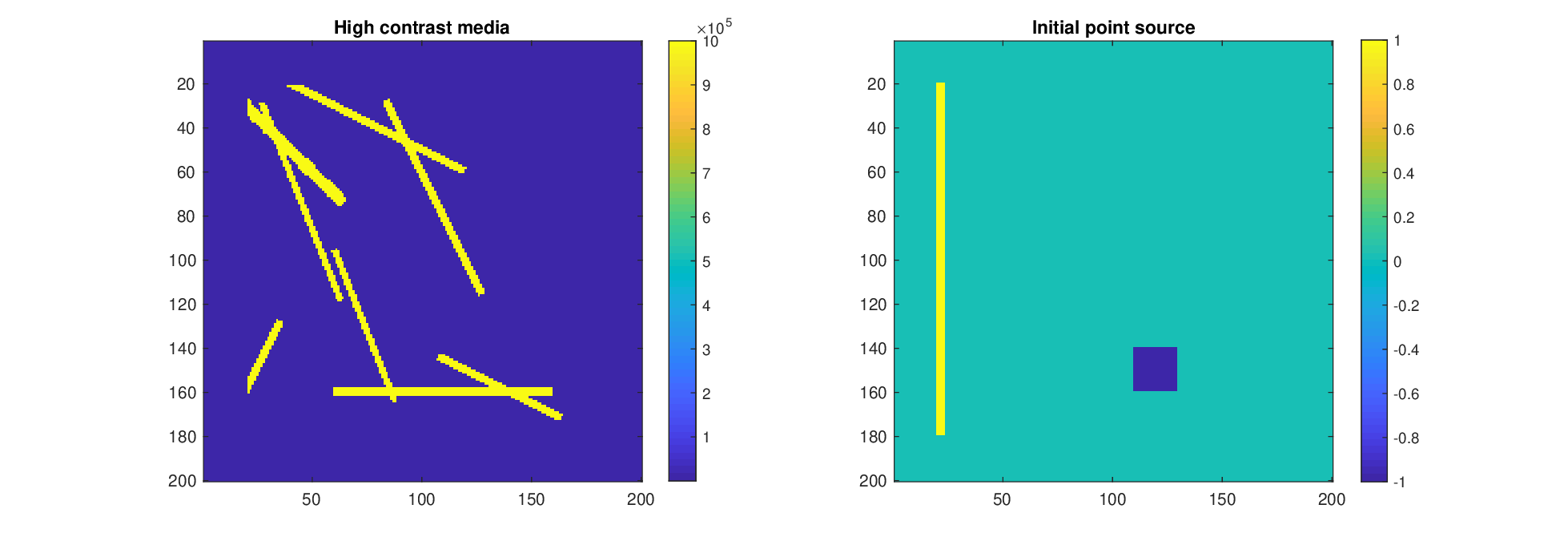}
\caption{The medium parameter $\kappa$}\label{fig:ex1_a_f}
\end{figure}

The fine and coarse mesh sizes are taken as $h = 1/200$ and $H = 1/20$, respectively. The time step size is set to $\Delta t = 0.1$ for Case 1 and $\Delta t = 0.002$ for Case 2.

We compare the convergence of the two-level preconditioned conjugate gradient (PCG) method equipped with different coarse spaces: the standard multiscale finite element space $V_{\text{ms}}$ ($\dim V_{\text{ms}} = 361$), the generalized multiscale finite element space $V_{\text{gms}}$ 
($\dim V_{\text{gms}} = 722$), the NLMC subspace $V_{H,1}$ that corresponds to the high permeability ($\dim V_{H,1} = 107$), the full NLMC space $V_H = V_{H,1} + V_{H,2}$ ($\dim V_H = 507$), and the space $V_H = V_{H,1} + V_{\text{ms}}$ ($\dim V_H = 468$). We let $cr=log(\frac{\kappa_{max}}{\kappa_{min}})$. The relative energy error convergence histories for different $cr$s are reported in Figures~\ref{fig:ex1_case1} and \ref{fig:ex1_case2}, for Case 1 and 2, respectively. As illustrated in Figure~\ref{fig:ex1_case1}, using the classical polynomial coarse space results in the slowest convergence, requiring approximately more than 50 iterations to reach the prescribed tolerance. When using MsFEM coarse space results, the iteration number is smaller than the case of classical polynomial, but the convergence is still contrast-dependent. The GMsFEM coarse space reduces the number of iterations to about 26, but at the expense of a larger coarse-space dimension. By contrast, the proposed NLMC space $V_H = V_{H,1} + V_{H,2}$ achieves the fastest convergence in only about 23 iterations, while maintaining a moderate dimension. In addition, the results indicates that enriching $V_{H,1}$ with $V_{\text{ms}}$ also produces a convergence behavior independent of contrast, and the number of iterations are slightly larger than those of GMsFEM.

Furthermore, Figure~\ref{fig:ex1_case2} shows that the NLMC subspace $V_{H,1}$ alone yields a convergence rate comparable to that of GMsFEM (approximately 30 iterations), however, $V_{H,1}$ has a much smaller dimension. This confirms that $V_{H,1}$ is particularly effective in preconditioning the high-permeability components of the stiffness matrix.  

	\begin{figure}
		\centering
			\includegraphics[width=0.24\textwidth]{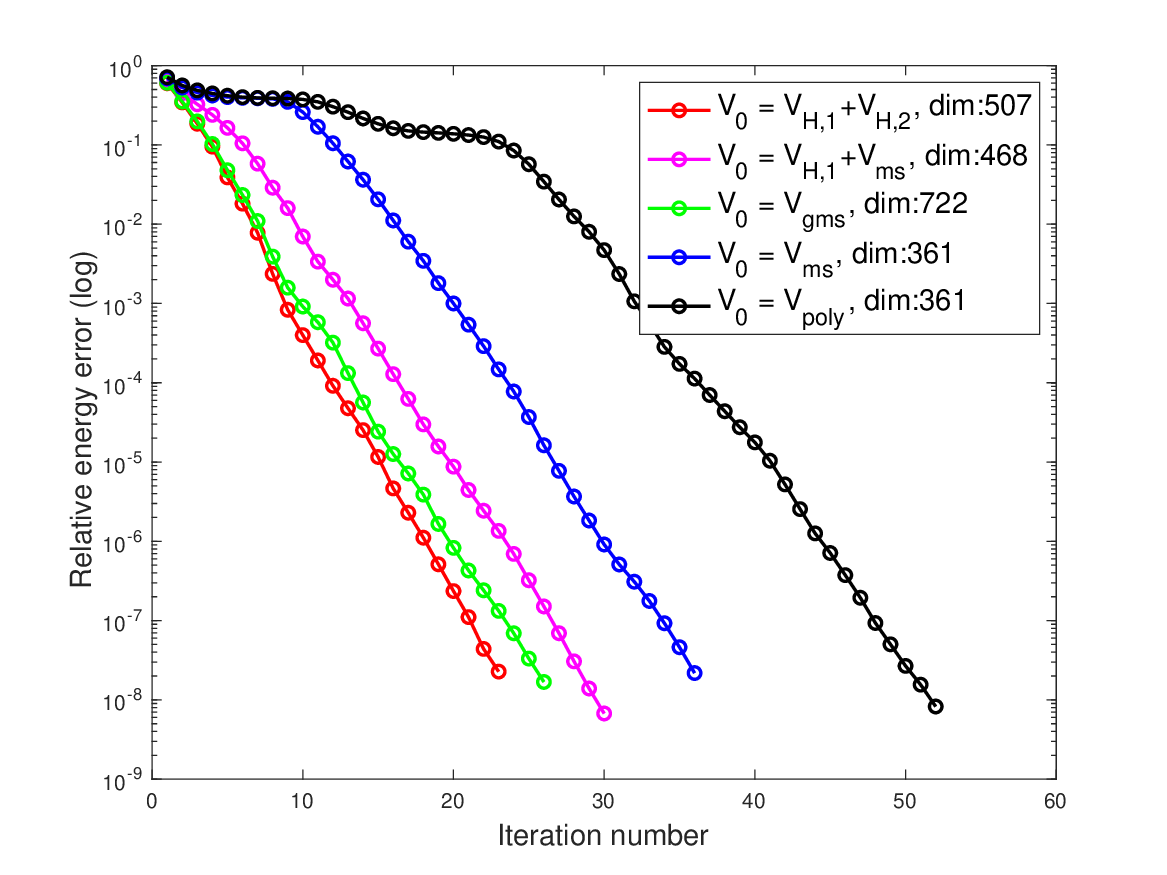}
			\includegraphics[width=0.24\textwidth]{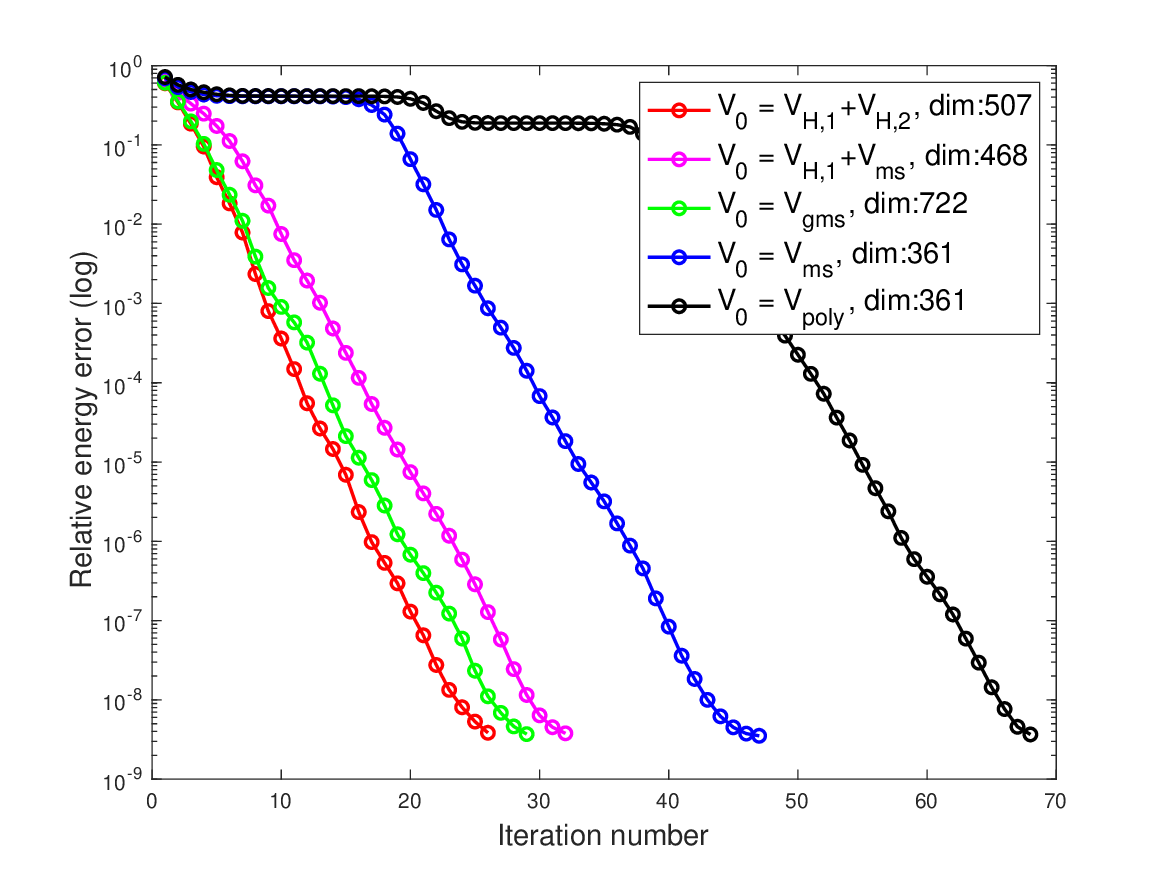}
			\includegraphics[width=0.24\textwidth]{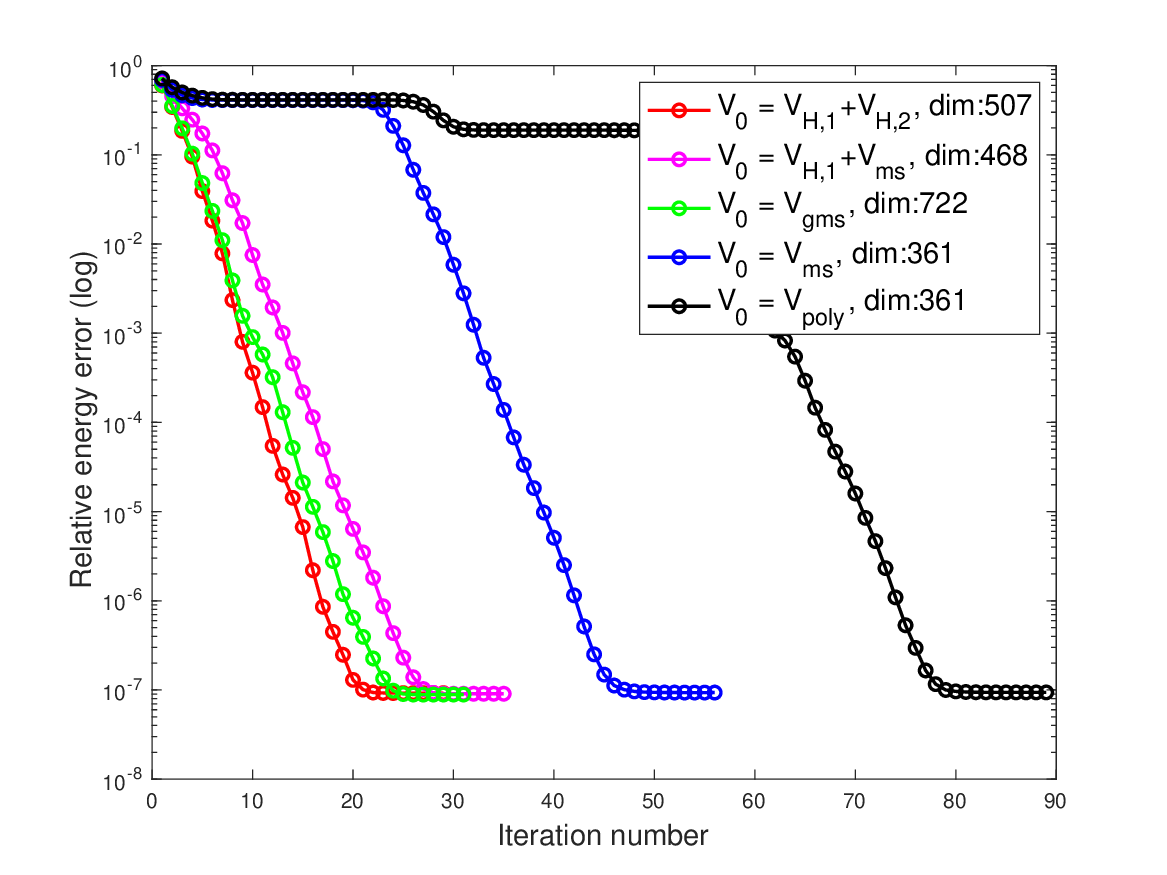}
            \includegraphics[width=0.24\textwidth]{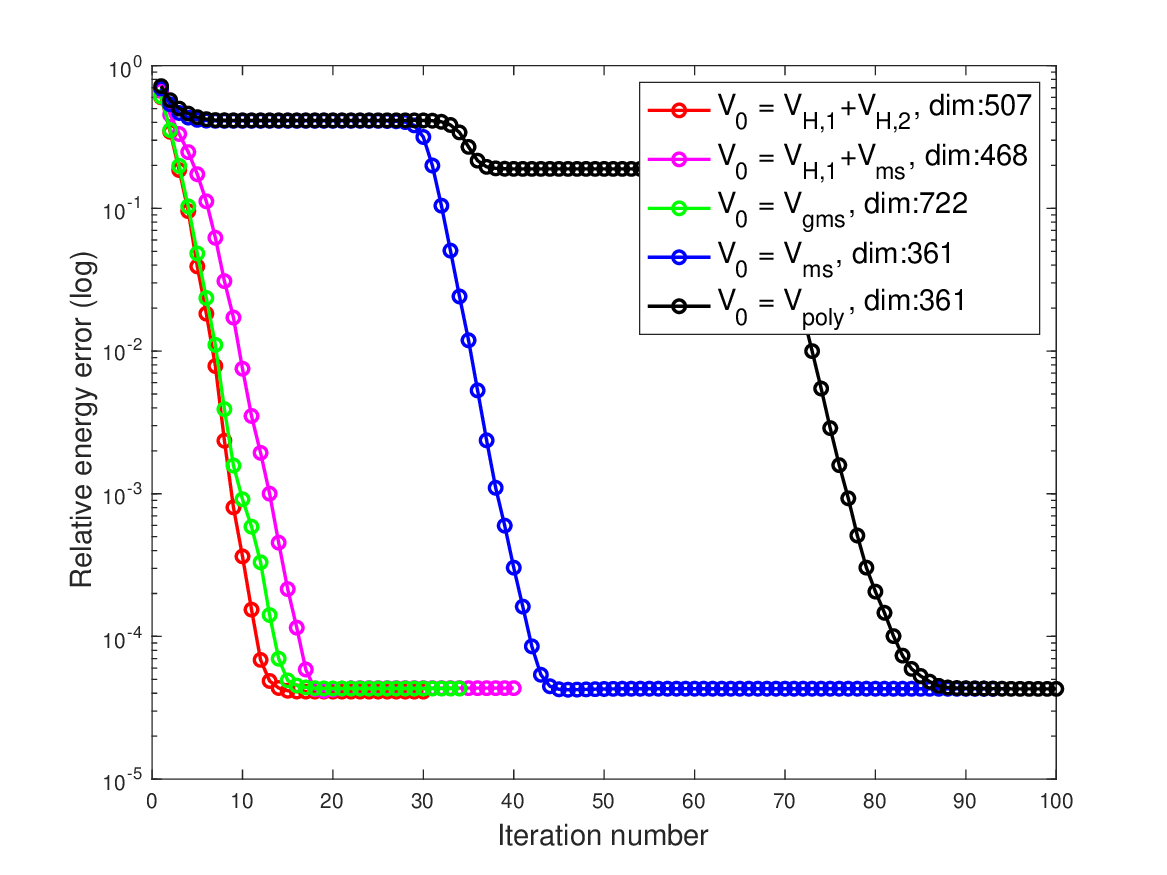}
		\caption{Energy error convergence history with PCG. $\text{dim}(V_{ms})= 361,\text{dim}(V_{gms})= 722, \text{dim}(V_{H,1}+V_{ms})=468, \text{dim}(V_{H,1}+V_{H,2})= 507$. From left to right: $cr = 4,6,8,10$.}\label{fig:ex1_case1} 
	\end{figure}

	\begin{figure}
		\centering
			\includegraphics[width=0.24\textwidth]{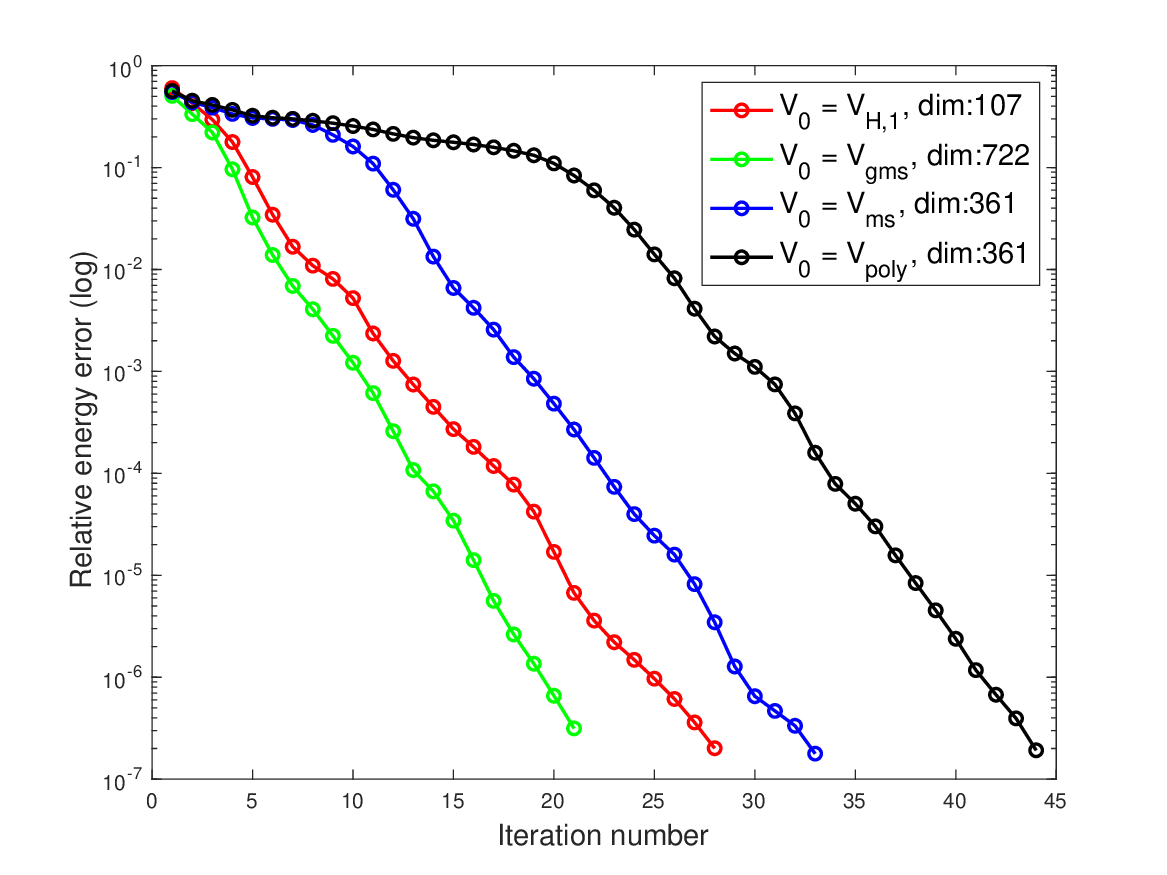}
			\includegraphics[width=0.24\textwidth]{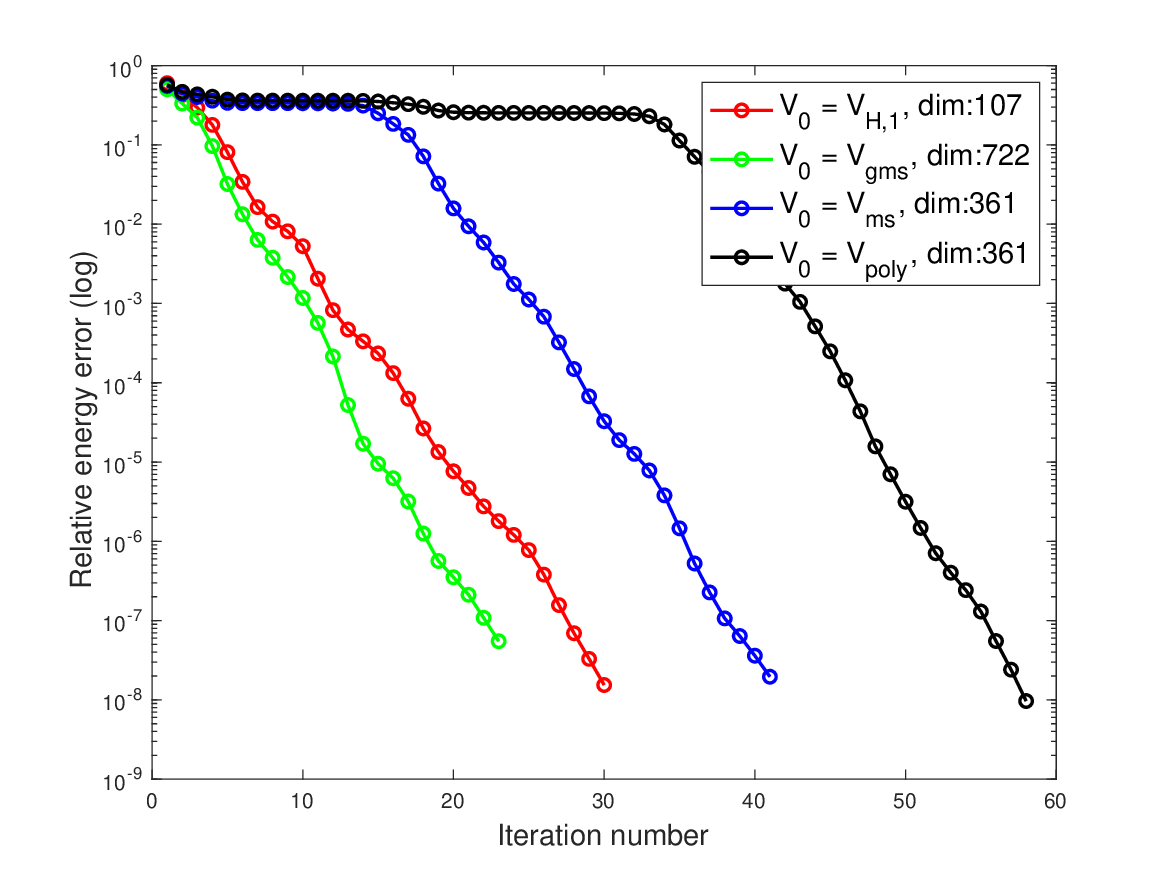}
			\includegraphics[width=0.24\textwidth]{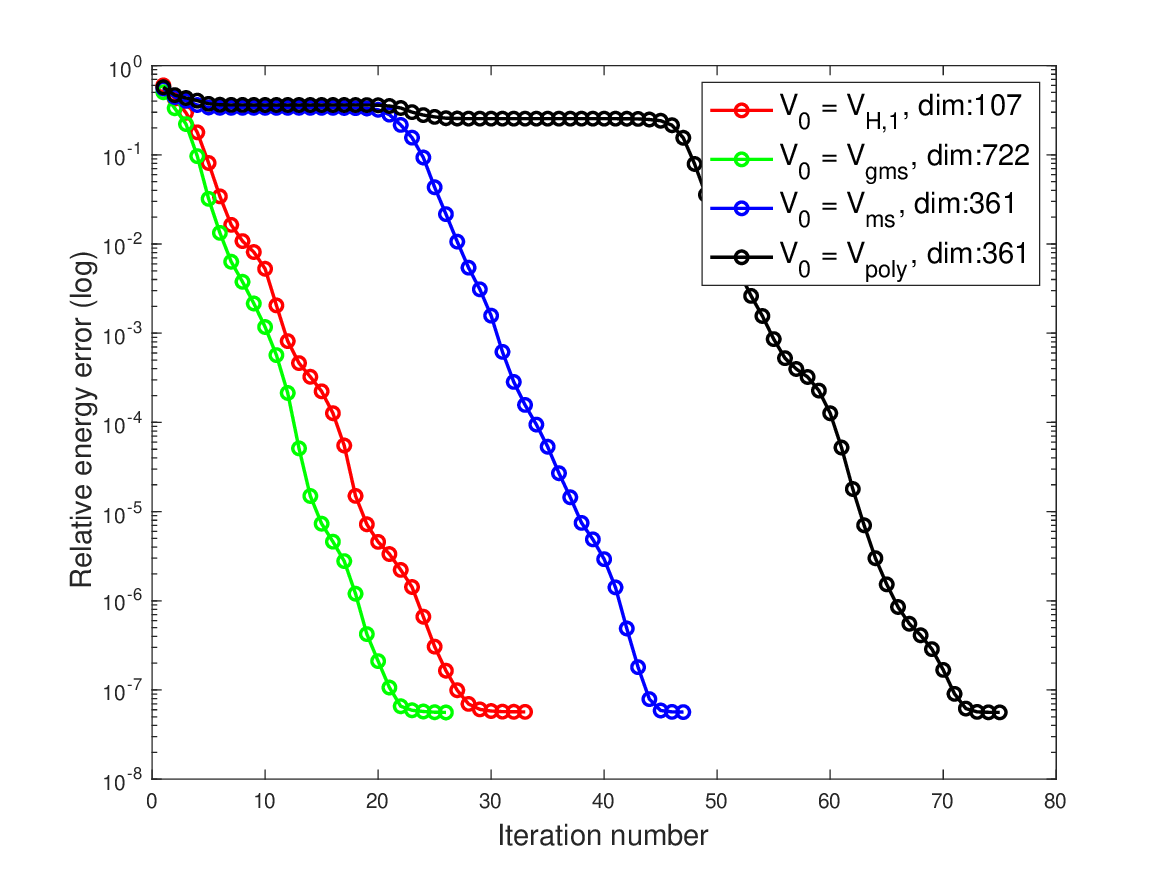}
			\includegraphics[width=0.24\textwidth]{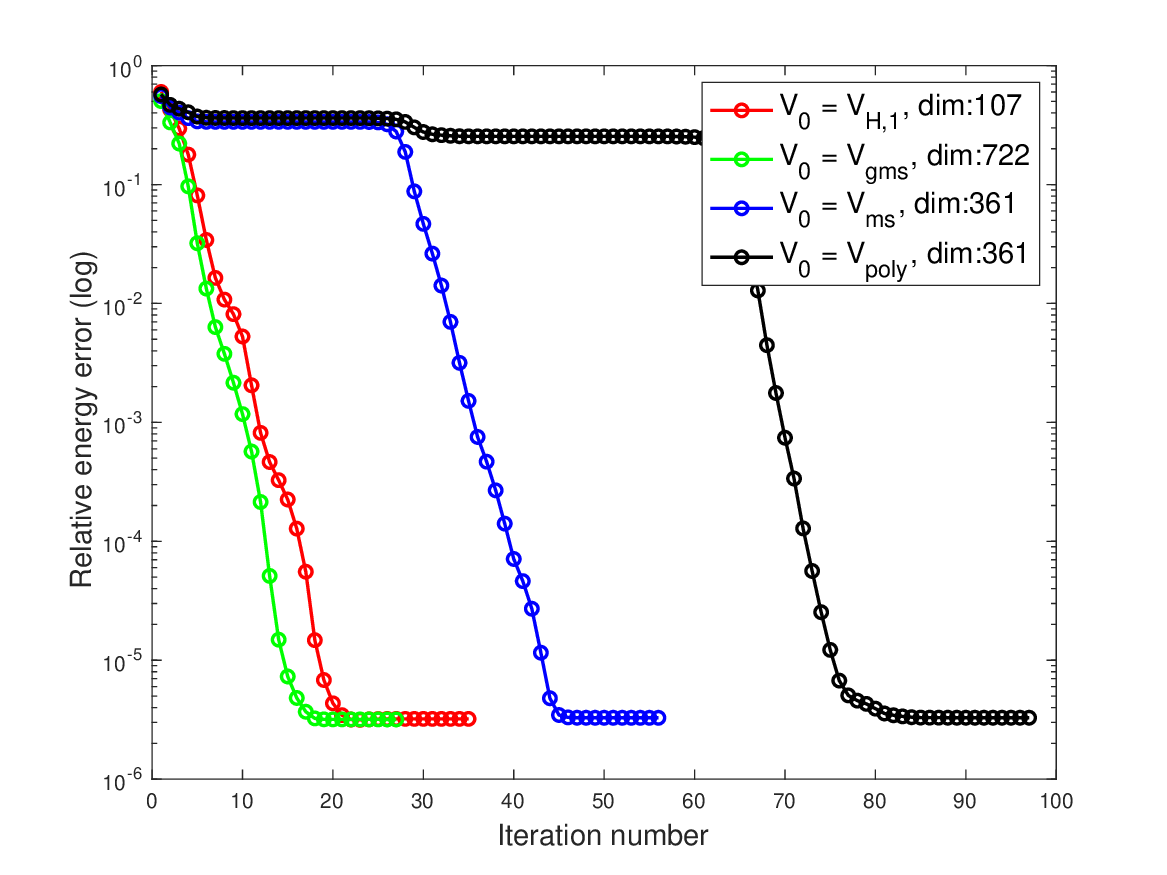}
		\caption{ Energy error convergence history with PCG. $\text{dim}{V_{ms}}= 361,\text{dim}{V_{gms}}= 722, \text{dim}{V_{H,1}}= 107$. From left to right: $cr = 4,6,8,10$. }\label{fig:ex1_case2}
	\end{figure}

These observations collectively demonstrate that $V_H = V_{H,1} + W_H$ or $V_{H,1}$ proposed in this work provides a good balance between computational cost and convergence robustness, making it a favorable coarse space for preconditioning the stiffness matrix in high-contrast multiscale problems. The detailed iteration number and computing time are reported in Tables \ref{tab:ex1_case1} and \ref{tab:ex1_case2} for Case 1 and 2, respectively. 

\begin{table}[htbp]
\centering
\small
\setlength{\tabcolsep}{4pt}
\caption{Two-level PCG solver performance: iteration time (seconds) and iteration number (in parentheses) for $N_x=20$, $n_x=10$, $\Delta t=0.1$.}\label{tab:ex1_case1} 
\begin{tabular}{cccccc}
\toprule
$C_r$ & $V_{H,1}+V_{H,2}$ & $V_{H,1}+V_{\text{ms}}$ & $V_{\text{ms}}$ & $V_{\text{gms}}$ & Poly \\
\midrule
4  & 5.01 (22) & 6.1 (27) & 7.1 (34) & 4.4 (23) & 10.0 (49) \\
6  & 5.4 (23) & 7.1 (30) & 8.6 (43) & 5.2 (26) & 13.3 (67) \\
8  & 5.2 (23) & 7.3 (33) & 10.4 (54) & 5.2 (26) & 18.2 (86) \\
10 & 5.3 (23) & 8.4 (39) & 12.7 (63) & 6.1 (31) & 22.2 (100) \\
\bottomrule
\end{tabular}
\end{table}

\begin{table}[htbp]
\centering
\small
\setlength{\tabcolsep}{5pt}
\caption{Two-level PCG solver performance: iteration time (seconds) and iteration number (in parentheses) for $N_x=20$, $n_x=10$, $\Delta t=0.002$.}\label{tab:ex1_case2} 
\begin{tabular}{ccccc}
\toprule
$C_r$ & $V_{H,1}$ & $V_{\text{ms}}$ & $V_{\text{gms}}$ & Poly \\
\midrule
4  & 7.4 (29) & 10.5 (34) & 6.1 (22) & 10.7 (45) \\
6  & 5.5 (31) & 7.7 (42) & 4.6 (24) & 10.7 (59) \\
8  & 7.4 (34) & 12.7 (48) & 6.3 (27) & 16.3 (76) \\
10 & 9.6 (36) & 17.6 (57) & 9.1 (28) & 26.5 (98) \\
\bottomrule
\end{tabular}
\end{table}

Moreover, In Figure ~\ref{fig:ex1_Viter}, we compare the convergence when the basis of $V_H = V_{H,1} + V_{H,2}$ are computed by a direct solver or by the iterative solver. We can see that, within 7 iterations, the NLMC basis constructed by the iterative method can achieve almost the same results as the ones constructed by the direct method. This indicates that we can improve the efficiency of NLMC basis construction significantly while maintaining the iteration counts in the two level preconditioning method.

\begin{figure}[ht!]
\centering
\includegraphics[width=0.24\textwidth]{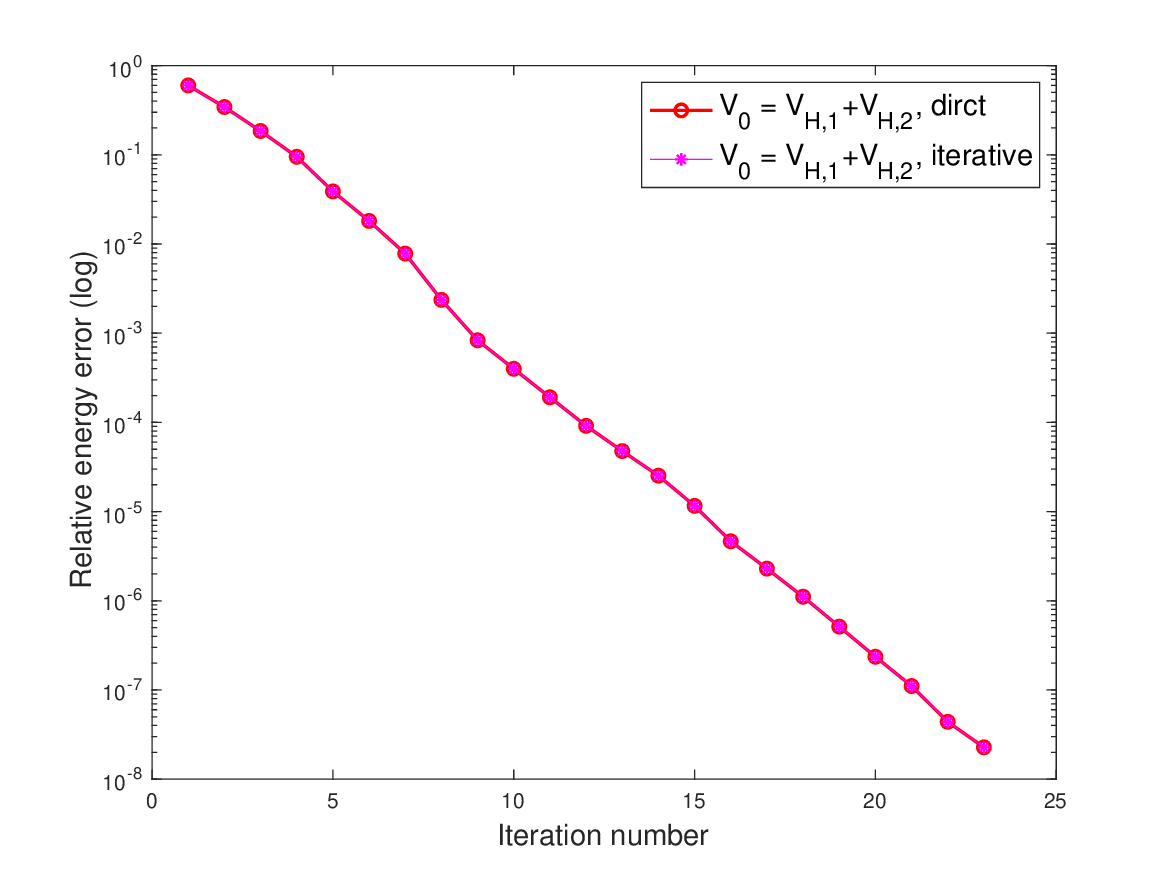}
\includegraphics[width=0.24\textwidth]{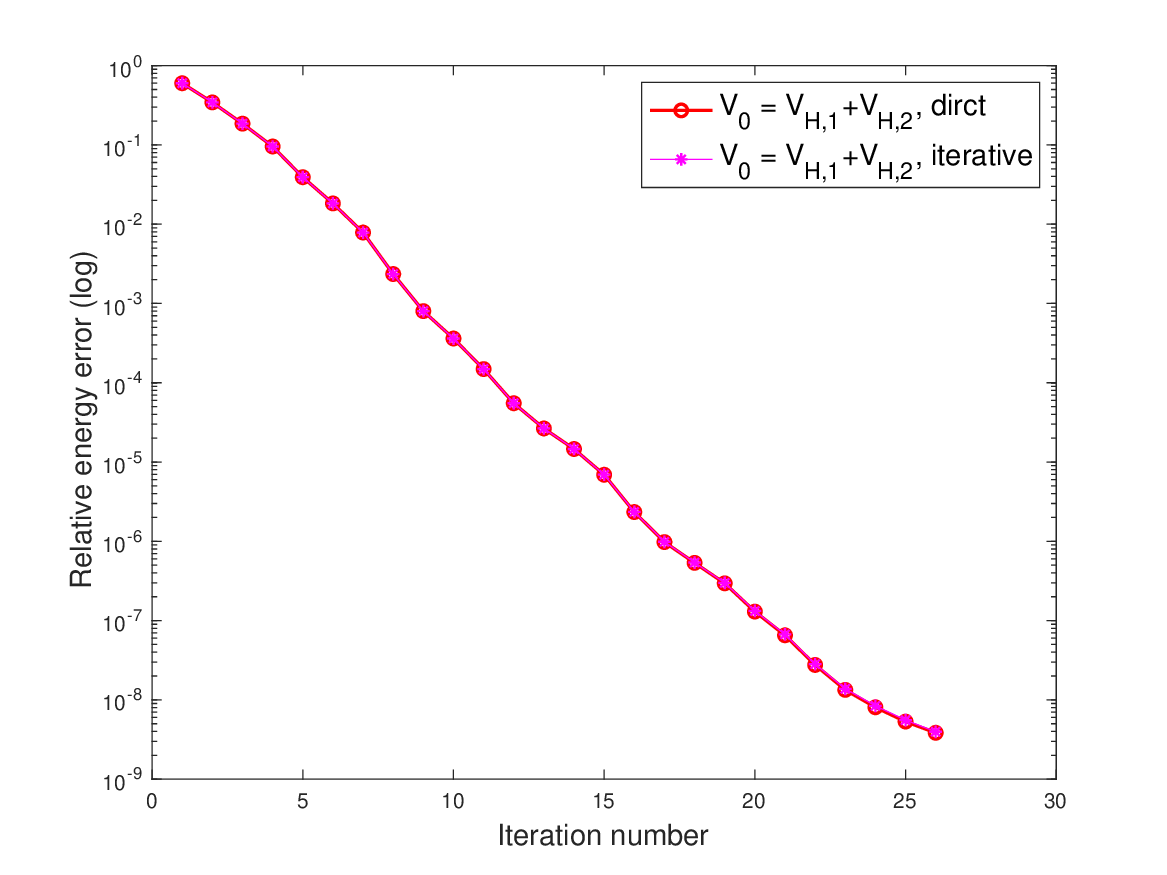}
\includegraphics[width=0.24\textwidth]{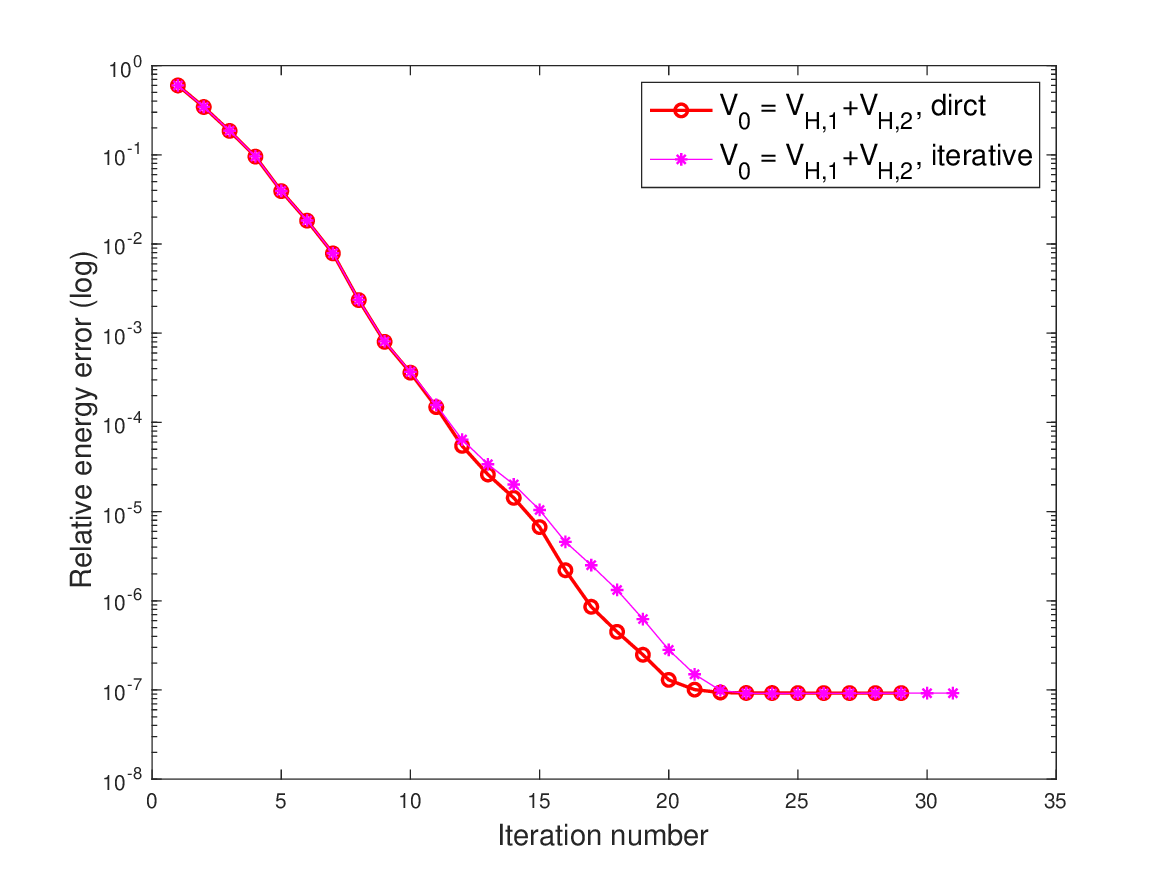}
\includegraphics[width=0.24\textwidth]{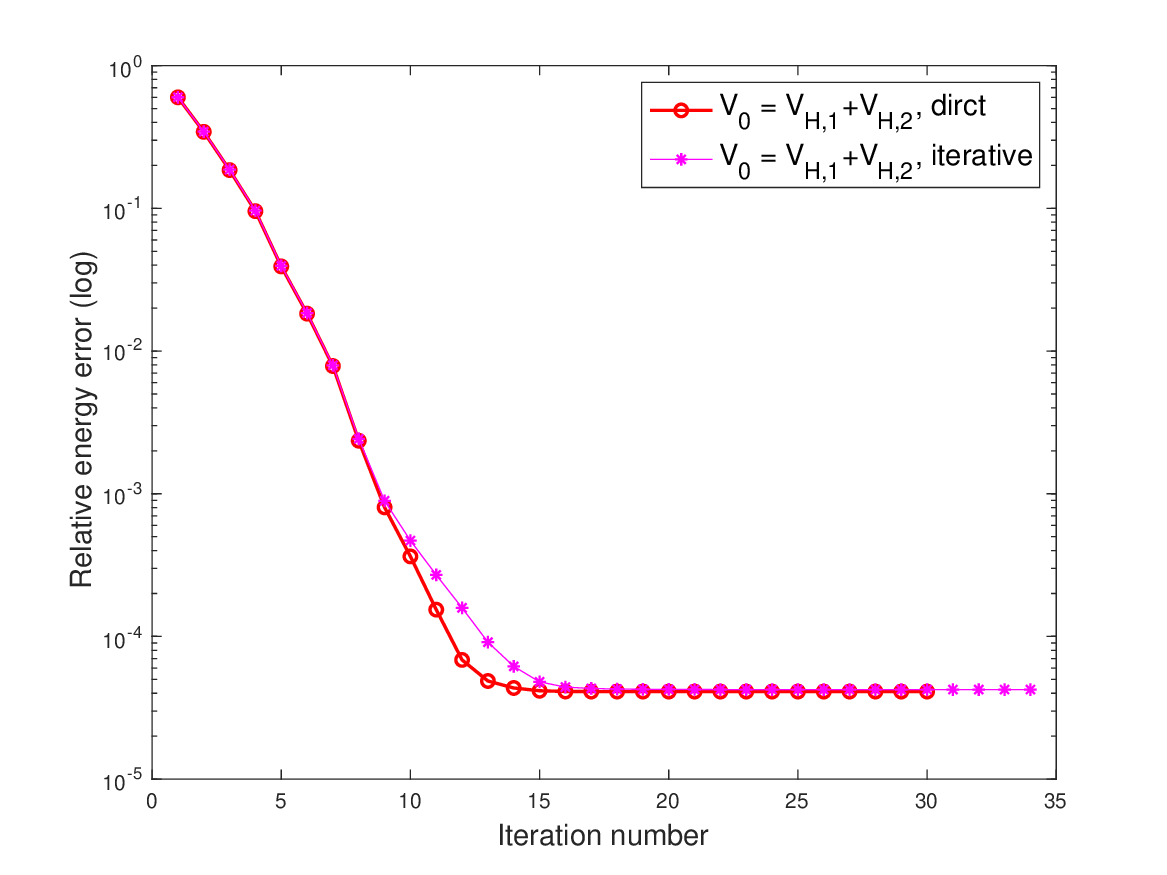}
\caption{Energy error history using coarse space $V_{H}=V_{H,1}+V_{H,2}$, where the basis functions are computed directly or by iterative scheme with 7 iteration. From left to right: $cr=4,6,8,10$.}\label{fig:ex1_Viter}
\end{figure}

\subsection{A time dependent permeability field}
In the second example, we consider a compressible fluid flow problem
\begin{equation*}
\begin{aligned}
    \partial_t(\phi \rho(u)) - \nabla \cdot (\frac{\kappa}{\mu} \rho(u)  \nabla u) &= f \;\;\;\; \text{ on } \Omega\times (0,T] \\
    u &= g \;\;\;\; \text{ on }  \partial \Omega  \times (0,T]\\
        u &= u_0 \;\;\;\; \text{ on }   \Omega  \times \{t=0\}
\end{aligned}
\end{equation*}
where the computational domain $\Omega = [0,64m]\times [0,64m]$. The coarse grid mesh size $H=1/10$ and fine grid mesh size $h=1/200$. The porosity $\phi = 500$, viscosity $\mu = 5cP$, and compressibility parameter $c = 10^{-8} Pa^{-1}$. The initial pressure $u_0 = 2.16 \times 10^7 Pa$, and the boundary is set to $g = 2.16 \times 10^7 Pa$ as well. The fluid density 
\begin{equation*}
    \rho(u)  = \rho_{\text{ref}} e^{c(u-u_{\text{ref}})},
\end{equation*}
where $\rho_{\text{ref}} = 850 kg/m^3$ is the reference density, and $u_{\text{ref}} = 2\times 10^{7} Pa$ is the reference pressure. The source term is $f(x,y)=\text{exp}(-\frac{(x-0.55L_x).^2+(y-0.55L_y).^2}{(0.01L_x)^2})\times 10^7$, where $L_x=L_y=64$.

To compute the reference solution, we adopt the Explicit-Implict-Null (EIN) method \cite{EIN, duchemin2022mars, shu_ein_ldg}. The idea is to add and subtract a term with constant diffusion coefficient $\nabla \cdot (\kappa_0 \nabla u)$ in the nonlinear equation, and then apply implicit time marching for the added linear term, and explicit time marching for nonlinear term and subtracted linear term, that is,
\begin{equation*}
\begin{aligned}
       &\frac{1}{\Delta t} {((u^{n+1}, v)-(u^{n}, v))}  + (\kappa_0 \nabla u^{n+1}, \nabla v)  \\
       &+ (\kappa(u^{n}) \nabla u^{n}, \nabla v) -(\kappa_0 \nabla u^{n}, \nabla v)   = (f,v).
\end{aligned}
\end{equation*}
In a matrix form, we have
\begin{equation}\label{eq:ein-mat}
\begin{aligned}
       & (M+\Delta t A_0) u^{n+1}   =  Mu^{n} +\Delta t F - \Delta t A(\kappa(u^{n}))u^{n} + \Delta t A_0 u^{n},
\end{aligned}
\end{equation} 
where $(A_0)_{ji}=\int_{\Omega} \kappa_0 \nabla \phi^0_i \cdot \phi^0_j$ and $\phi^0_i$ are finite element basis functions.

The reference solutions is computed on a $201\times 201$ grid, using the EIN method. The coarse grid is $10\times 10$. The absolute permeability $\kappa$ in this example is shown in Figure \ref{fig:sol_perm_ex2}, the values in the yellow region (channels) are $10^{cr}$ where $cr=6, 7, 8$, and the blue region is $1$. To coarse space is constructed using the absolute permeability $\kappa$ (which is $\kappa_0$ in equation \eqref{eq:ein-mat}) and stays unchanged during the simulation.

	\begin{figure}
		\centering
			\includegraphics[width=0.65\textwidth]{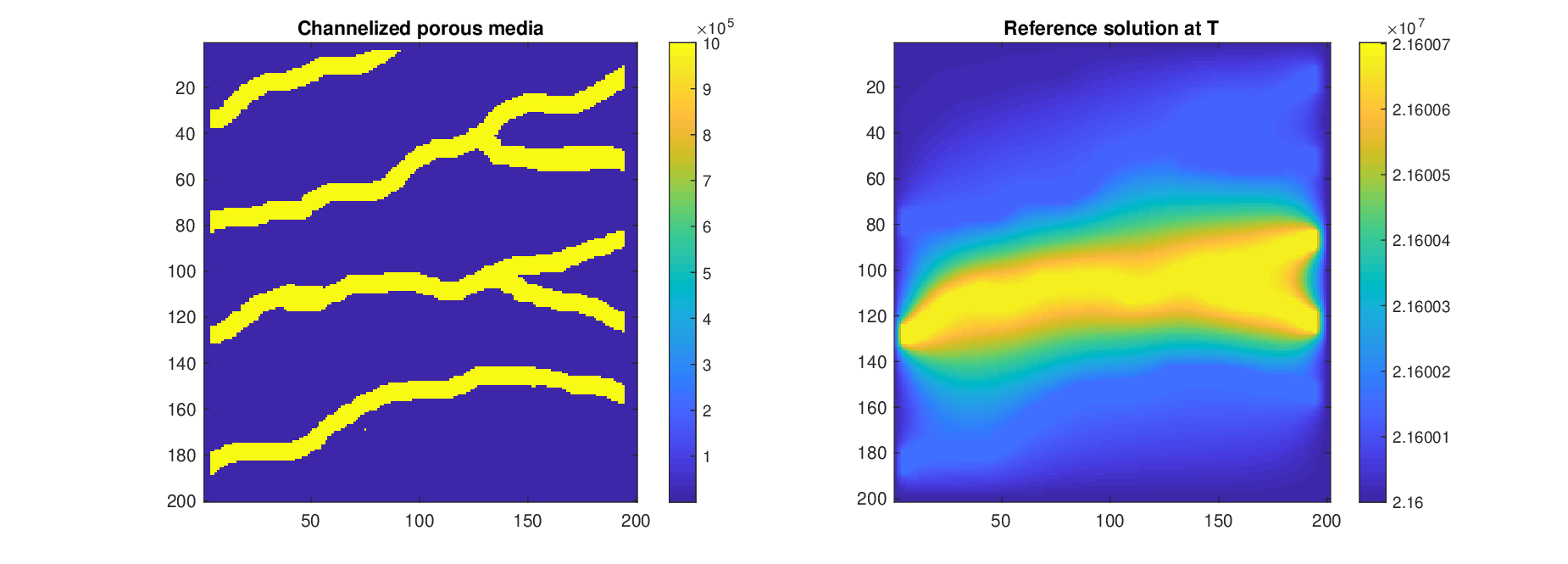}
		\caption{Left: the channelize porous media. Right: the reference solution at $T$.}\label{fig:sol_perm_ex2}
	\end{figure}

Similarly as before, for Case 1, we observe from Figure \ref{fig:ex2_V12} and Table \ref{tab:ex2_V12} that the using the coarse space $V_{gms}$, $V_{H,1}+V_{H,2}$ or $V_{H,1}+V_{ms}$ result in contrast-independent convergence rate, and $V_{H,1}+V_{H,2}$ requires much smaller number of iterations. Moreover, $V_{H,1}+V_{H,2}$ $V_{gms}$ are more robust with respect to the contrast ratio, but the dimension of the coarse space for $V_{H,1}+V_{H,2}$ is $158$, less than the dimension of the coarse space for $V_{gms}$ which is 242 in this example. 

	\begin{figure}
		\centering
			\includegraphics[width=0.3\textwidth]{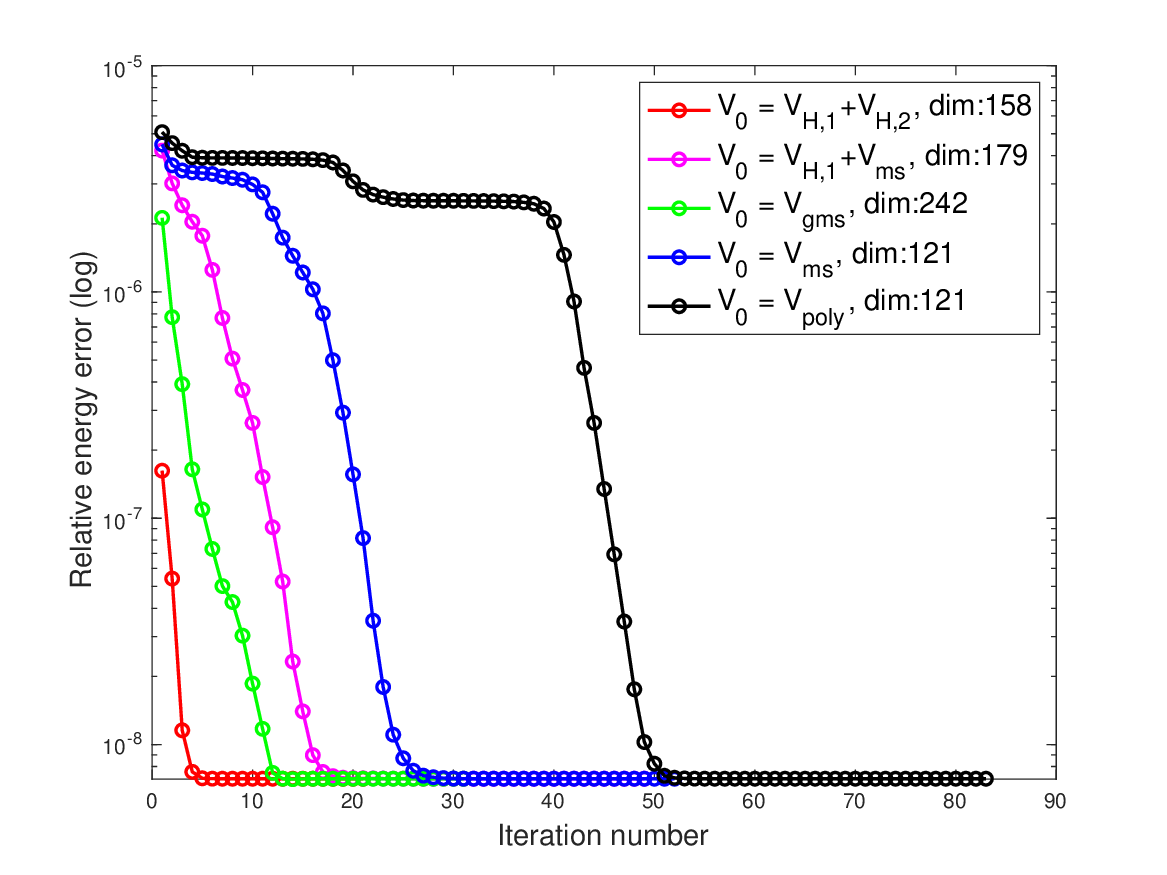}
			\includegraphics[width=0.3\textwidth]{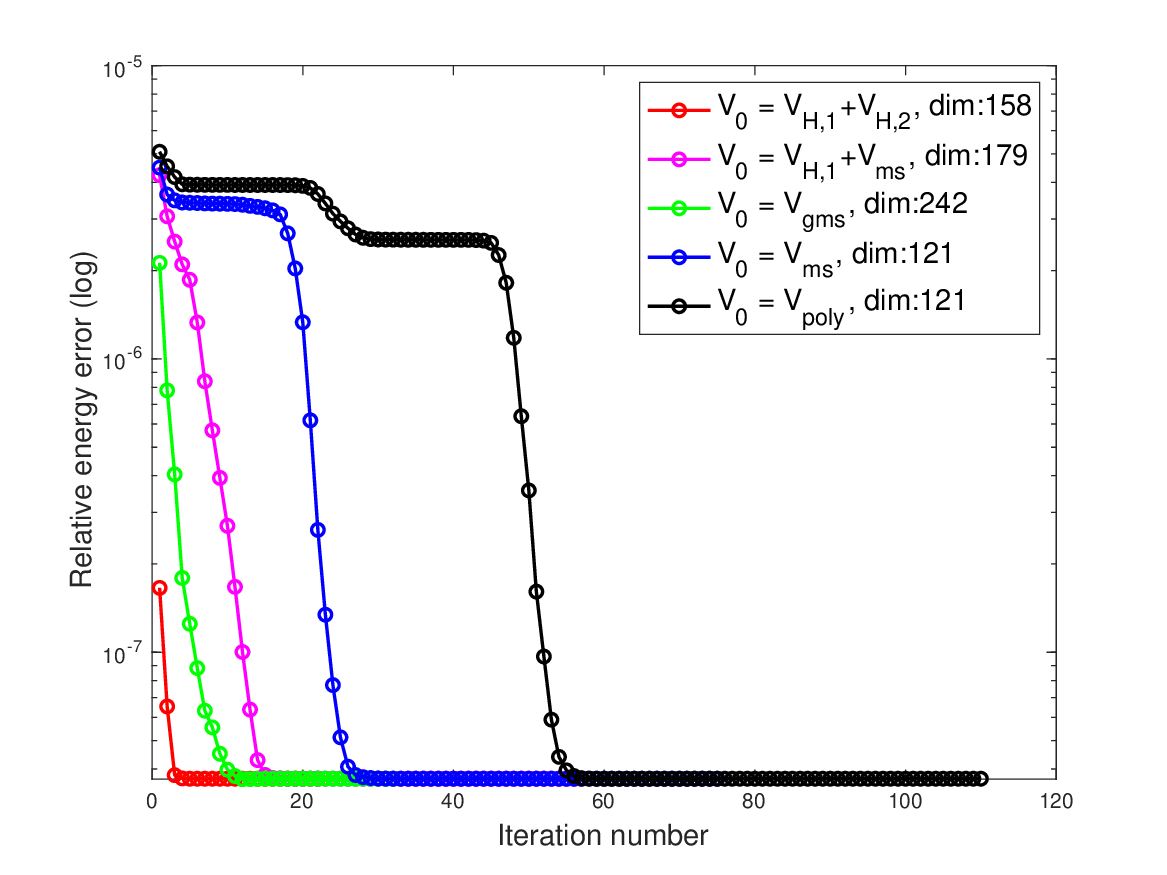}
			\includegraphics[width=0.3\textwidth]{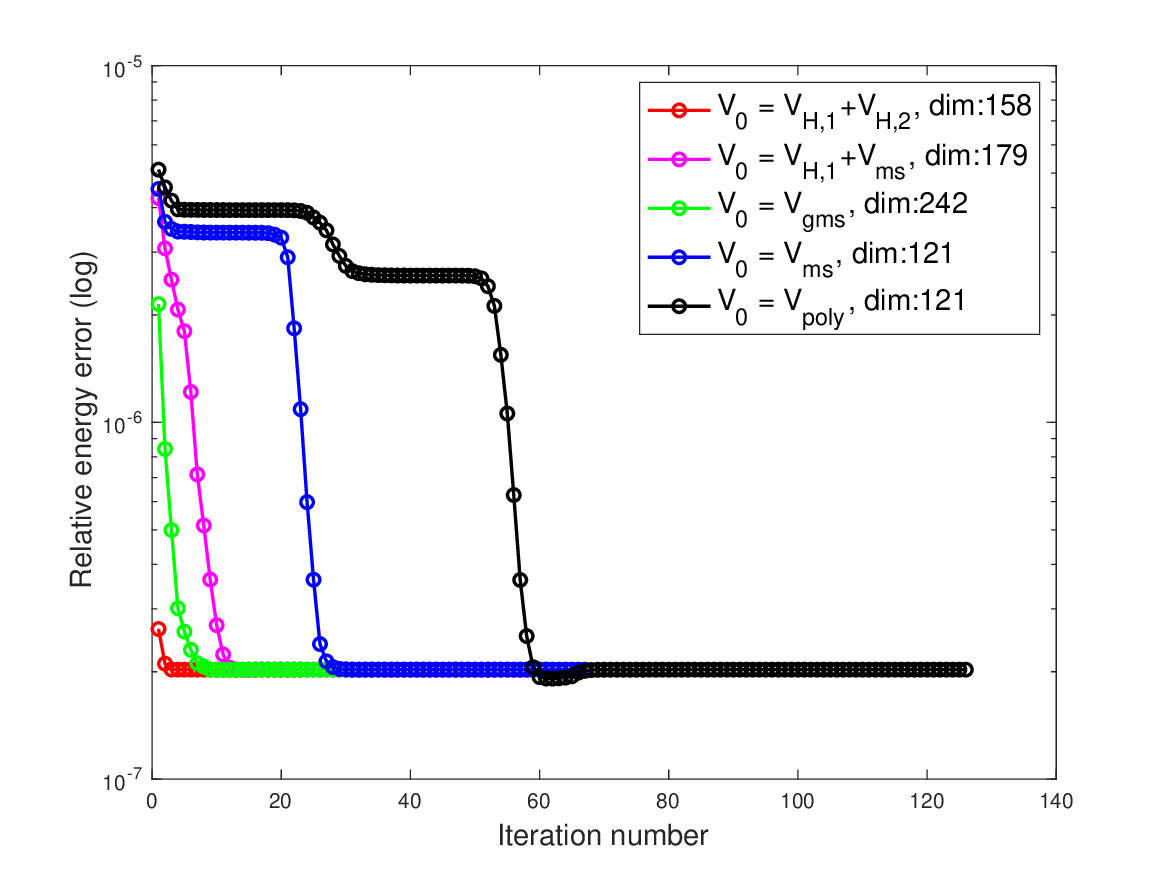}
		\caption{ Energy error convergence history with two level PCG. $\text{dim}{V_{ms}}= 121,\text{dim}{V_{gms}}= 242, \text{dim}{V_{H,1}+V_{ms}}=179, \text{dim}{V_{H,1}+V_{H,2}}= 158$. From left to right, $cr=6,7,8$.}\label{fig:ex2_V12}
	\end{figure}

\begin{table}[htbp]
\centering
\small
\setlength{\tabcolsep}{4pt}
\caption{Two-level PCG solver performance for $N_x=20$, $n_x=10$, $\Delta t=5\times 10^{-2}$. Iteration time in seconds and iteration number.}\label{tab:ex2_V12}
\begin{tabular}{cccccc}
\toprule
$C_r$ & $V_{H,1}+V_{H,2}$ & $V_{H,1}+V_{\text{ms}}$ & $V_{\text{ms}}$ & $V_{\text{gms}}$ & Poly \\
\midrule
6 & 2.0 (22)  & 4.5  (50) & 5.0 (53) & 3.1(32) & 7.5 (84) \\
7  &2.3 (23) & 5.9  (58) & 7.2 (76) & 3.1(33) & 10.8 (111) \\
8  &2.3  (24) & 5.9 (62) & 6.1 (68) & 3.2 (35) & 11.6 (127) \\
\bottomrule
\end{tabular}
\end{table}

For case 2, the results are presented in Figure \ref{fig:ex2_V1} and Table \ref{tab:ex2_V1}. It can be seen that with smaller time step size, only using $V_{H,1}$ as the coarse space can result in the contrast-independent convergence. Both $V_{H,1}$ and $V_{gms}$ are more robust with respect to the contrast ratio, compared with $V_{ms}$ and standard polynomial cases. While $V_{H,1}$ requires a little bit more iterations to converge, the dimension of the coarse space for $V_{H,1}$ is only $57$ which is much less than the dimension of the coarse space for $V_{gms}$.

	\begin{figure}
		\centering
			\includegraphics[width=0.3\textwidth]{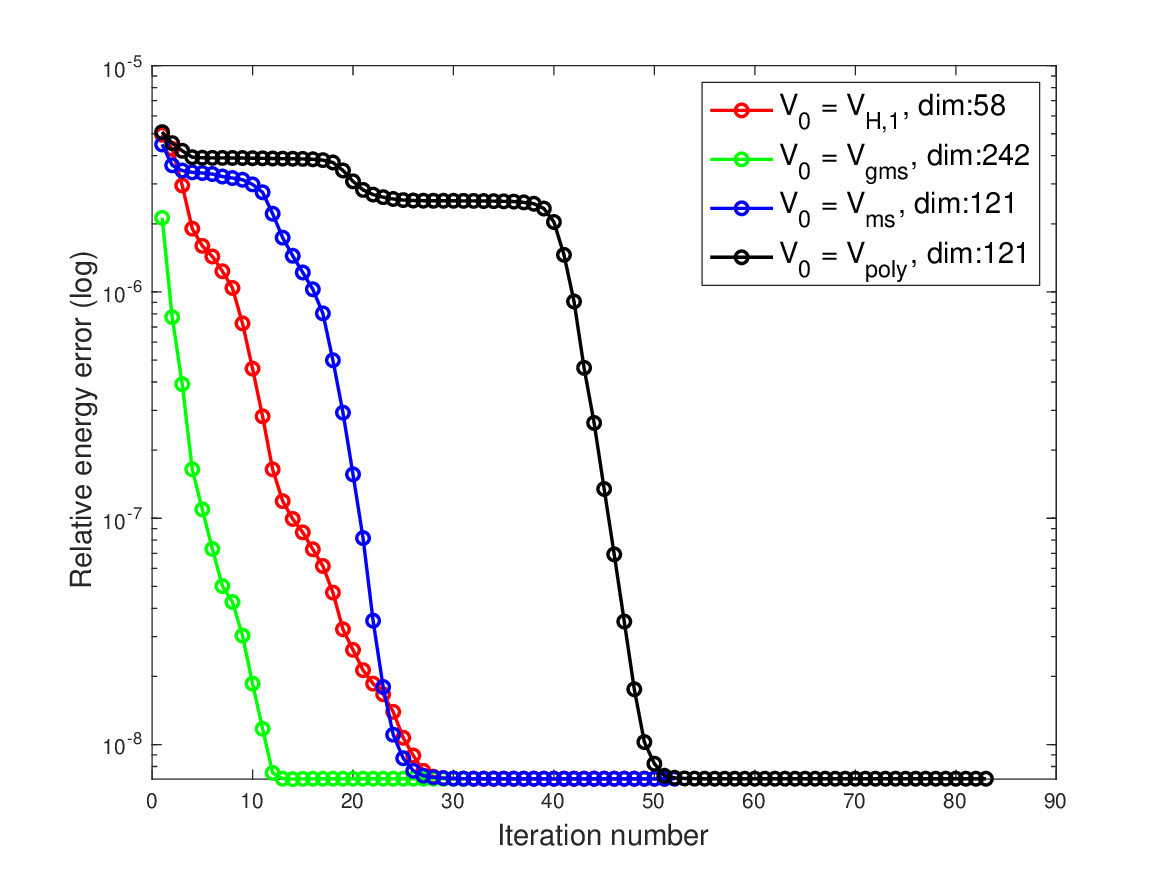}
			\includegraphics[width=0.3\textwidth]{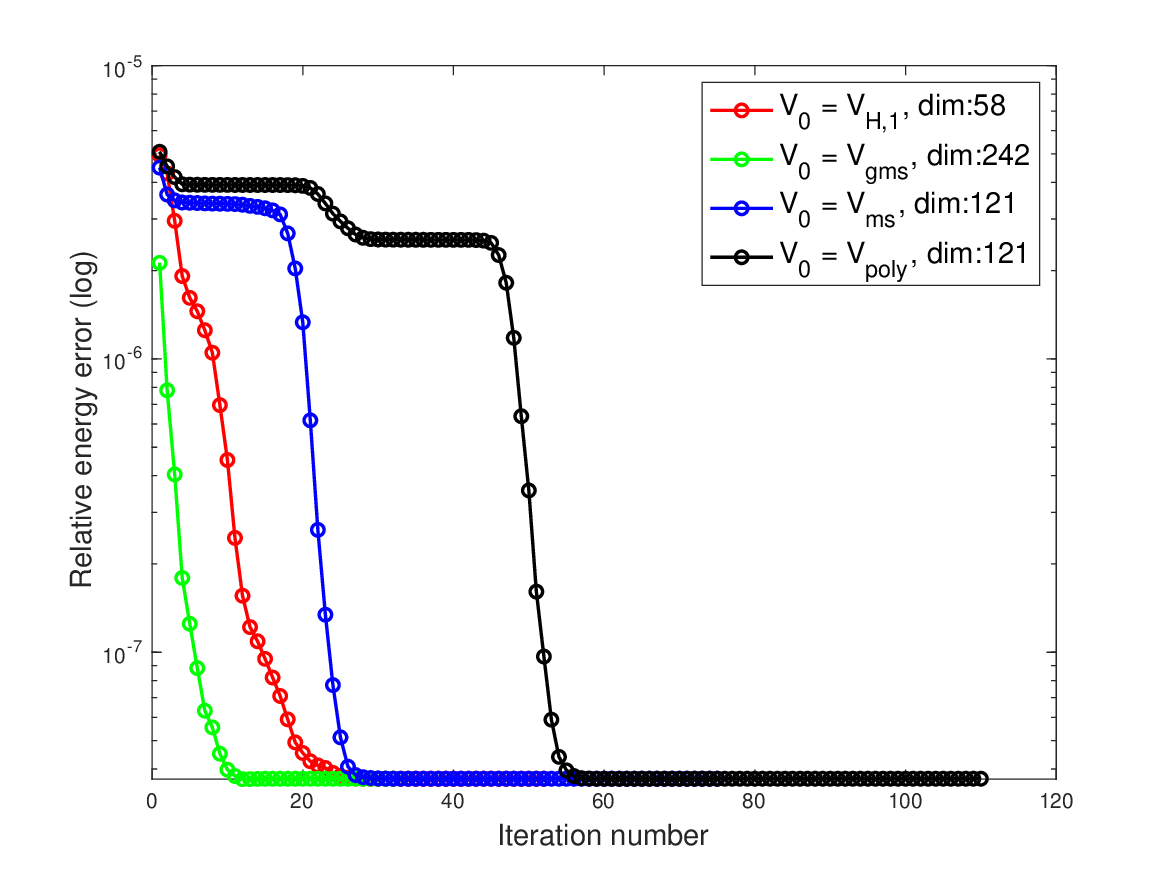}
			\includegraphics[width=0.3\textwidth]{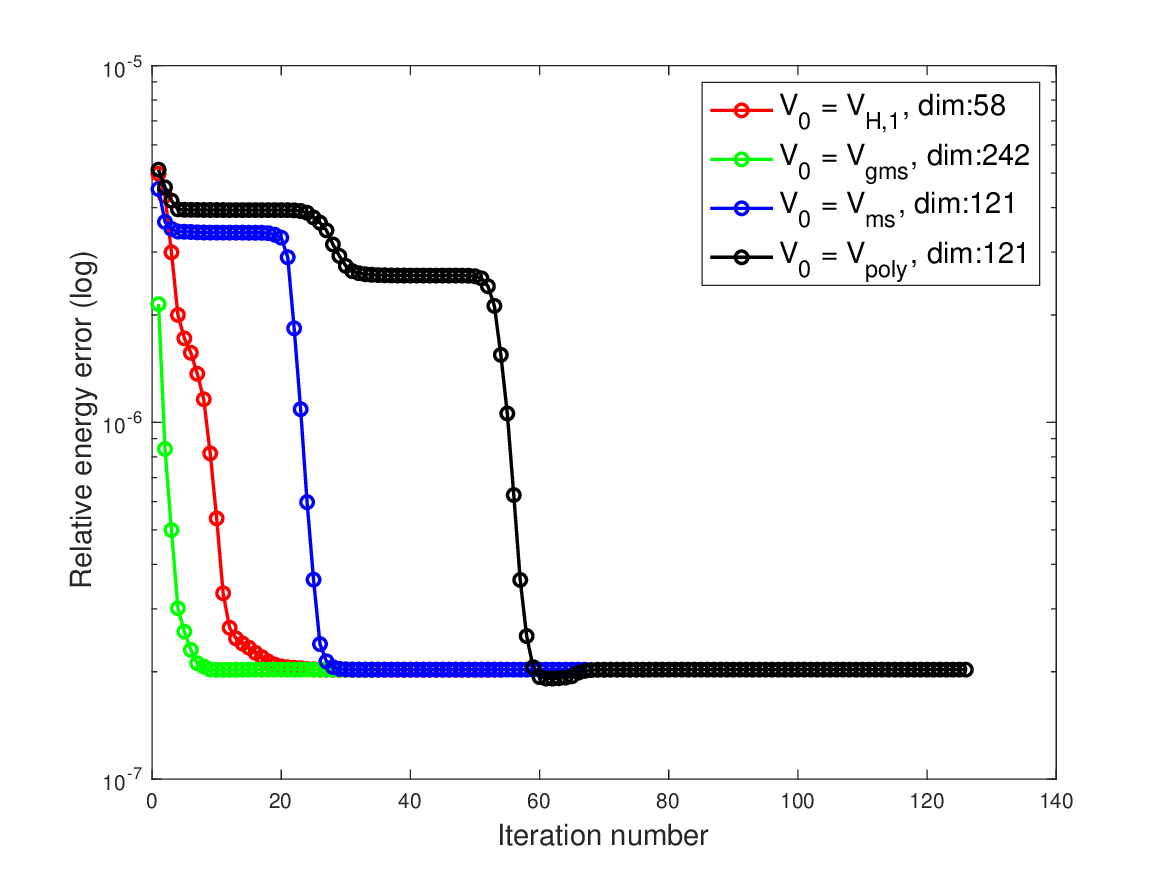}
		\caption{ Energy error convergence history with PCG. $\text{dim}{V_{ms}}= 121,\text{dim}{V_{gms}}= 242, \text{dim}{V_{H,1}}=58$.}\label{fig:ex2_V1}
	\end{figure}

\begin{table}[htbp]
\centering
\small
\setlength{\tabcolsep}{5pt}
\caption{Two-level PCG solver performance for $N_x=20$, $n_x=10$, $\Delta t=5\times 10^{-2}$. Iteration time in seconds and iteration number.}
\label{tab:ex2_V1}
\begin{tabular}{ccccc}
\toprule
$C_r$ & $V_{H,1}$ & $V_{\text{ms}}$ & $V_{\text{gms}}$ & Poly \\
\midrule
6 & 4.7 (52) & 4.6 (53) & 2.9 (32) & 7.2 (84) \\
7 & 5.3 (53) & 7.0 (76) & 3.1 (33) & 11.1 (111) \\
8 & 6.8 (55) & 8.6 (68) & 3.9 (35) & 12.6 (127) \\
\bottomrule
\end{tabular}
\end{table}

\subsection{Three dimensional case}
In this section, we test our proposed method on a three dimensional example. The channelized media and the reference solution are shown in Figure \ref{fig:ex3_kappa_u}. The computational domain $\Omega = [0,1]^3$. We used is PETSc in our implementation \cite{Balay2022a}.

\begin{figure}[ht!] \label{fig:ex3d_kappa_sol}
\centering
\includegraphics[width=0.45\textwidth]{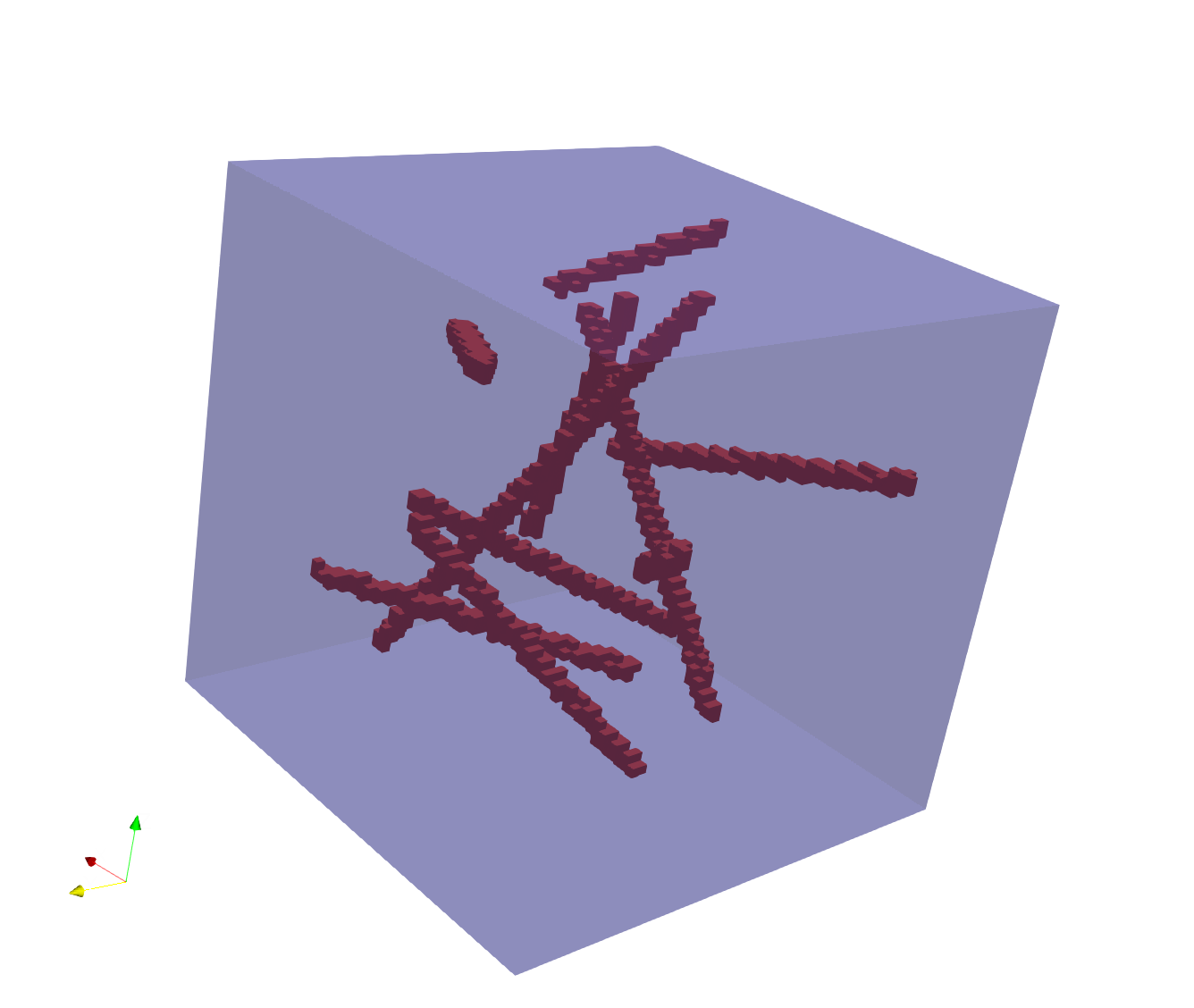}
\includegraphics[width=0.45\textwidth]{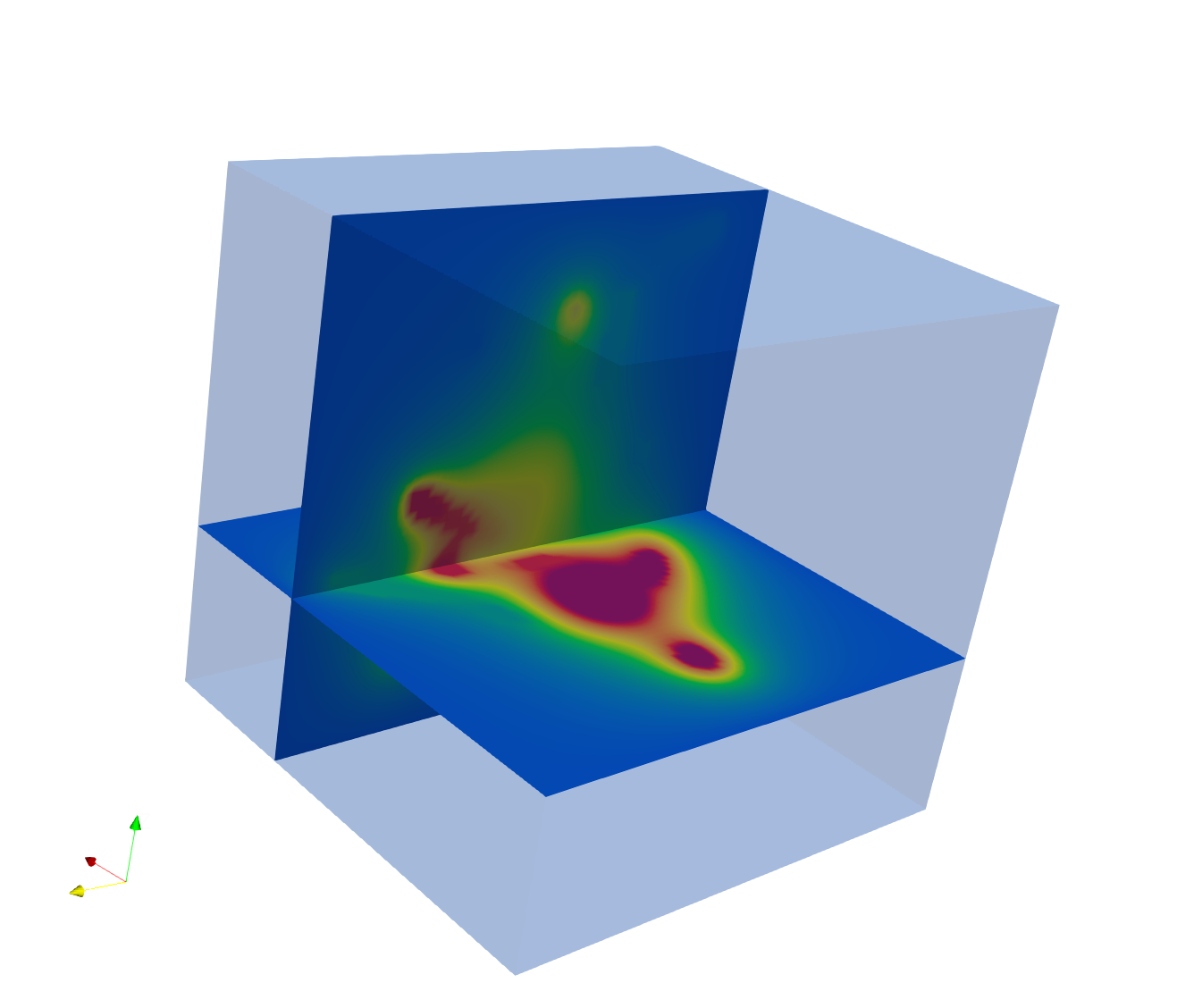}
\caption{The medium parameter $\kappa$ in three dimensional case, and the solution slice at final time.}\label{fig:ex3_kappa_u}
\end{figure}

The numerical results demonstrate that the proposed NLMC-based two-level preconditioners outperform the standard MsFEM, and the algebraic multigrid (gamg) solvers in both robustness and computational efficiency, see Figure \ref{fig:ex3_V12} and \ref{fig:ex3_V1} for case 1 and case 2, respectively. The coarse spaces $V_{H,1}+V_{H,2}$ and $V_{H,1}$ exhibit iteration counts that are essentially independent of the contrast ratio $cr$. Specifically, we can see from Table \ref{tab:pcg_iters_V12} and \ref{tab:pcg_iters_V1}, the full NLMC space $V_{H,1}+V_{H,2}$ requires merely $32$, $27$, $19$, and $13$ iterations for $C_r=4,6,8,10$, respectively, while the subspace $V_{H,1}$ requires $45$, $41$, $30$, and $21$ iterations. This is in contrast to the MsFEM and gamg approaches, whose iteration counts grow substantially with increasing contrast ratio, reaching $216$ and $121$ iterations, respectively, and eventually fail to converge at $cr=10$. 

Moreover, NLMC coarse space is comparable to GMsFEM in robustness, while it can have much less DOFs when the volume fraction of the channels is low. The NLMC coarse spaces achieve good reduction in problem dimension: the $V_{H,1}$ space contains only $130$ degrees of freedom, compared $1209$ for GMsFEM, while the $V_{H,1}+V_{H,2}$ space still maintains a competitive size of $1130$ degrees of freedom. The results in Table \ref{tab:strong_scalability} and \ref{tab:weak_scalability} also validate the \textit{weak scalability} of our method, as the iteration counts remain uniformly bounded across different coarse mesh resolutions ($H=1/10, 1/12, 1/14, 1/15$), and its \textit{strong scalability}, as the solver performance is robustly insensitive to the heterogeneity strength characterized by the contrast ratio.

\begin{figure}[ht!] \label{fig:ex3d_nf60_V12}
\centering
\includegraphics[width=0.24\textwidth]{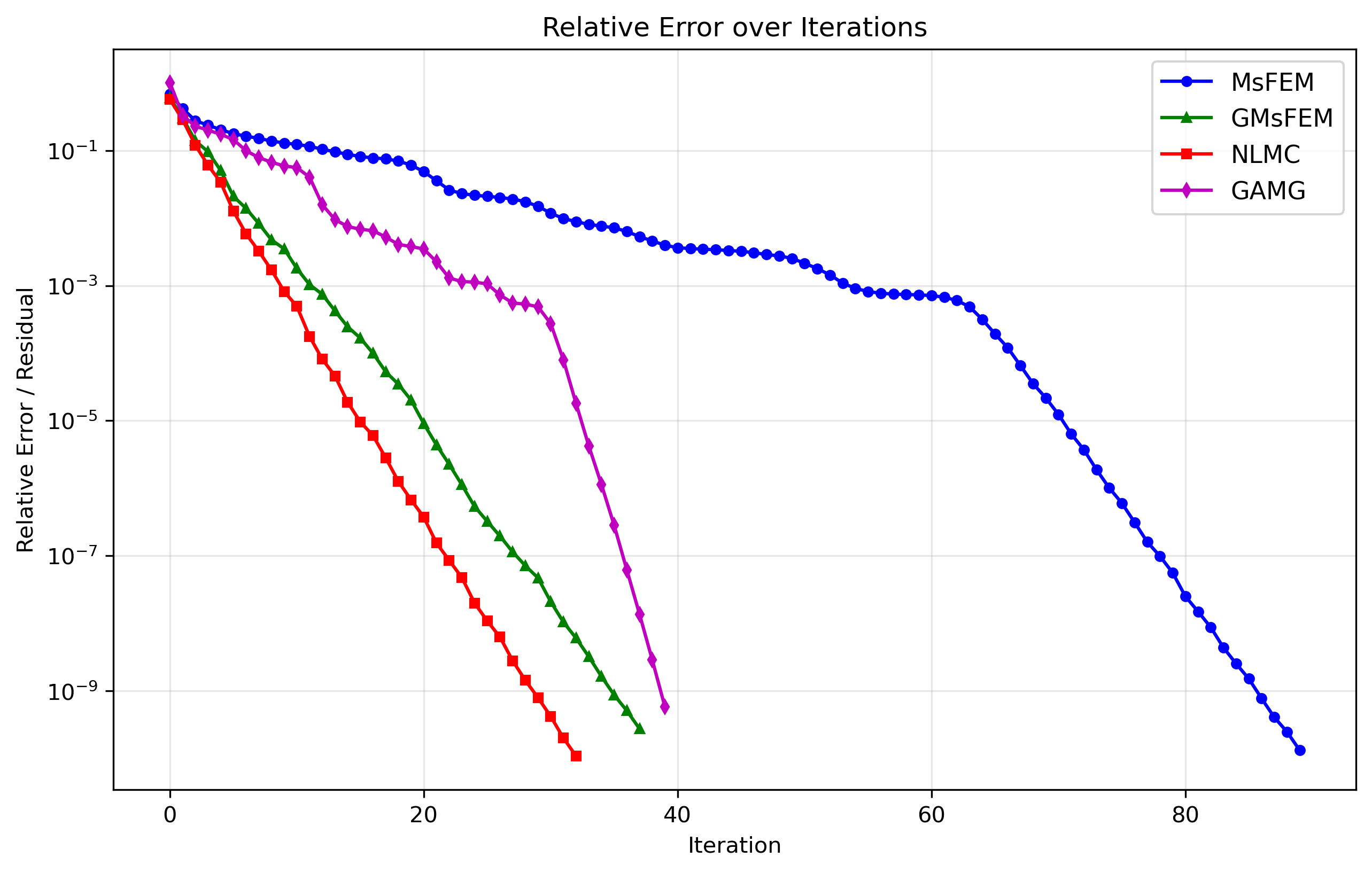}
\includegraphics[width=0.24\textwidth]{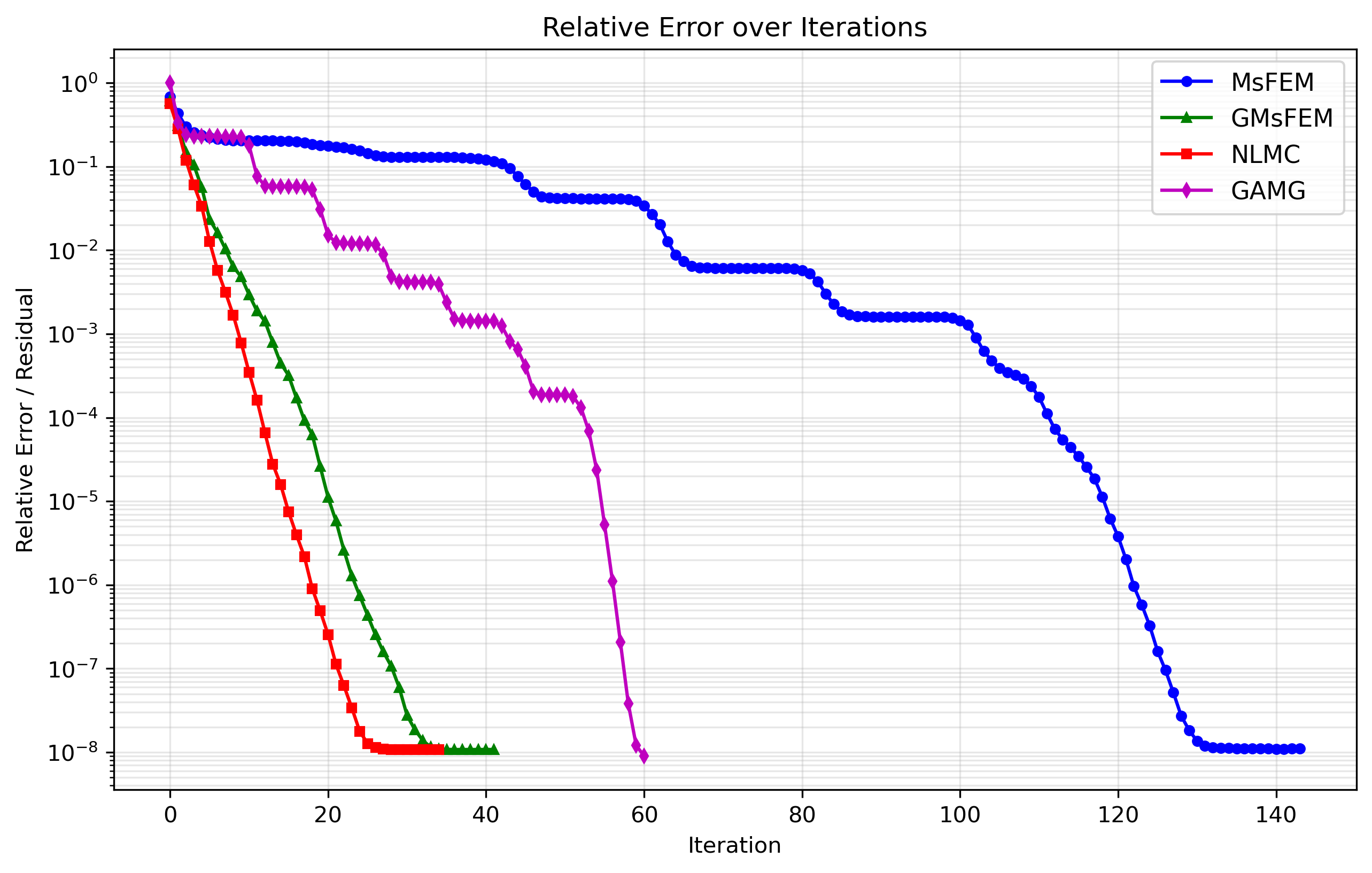}
\includegraphics[width=0.24\textwidth]{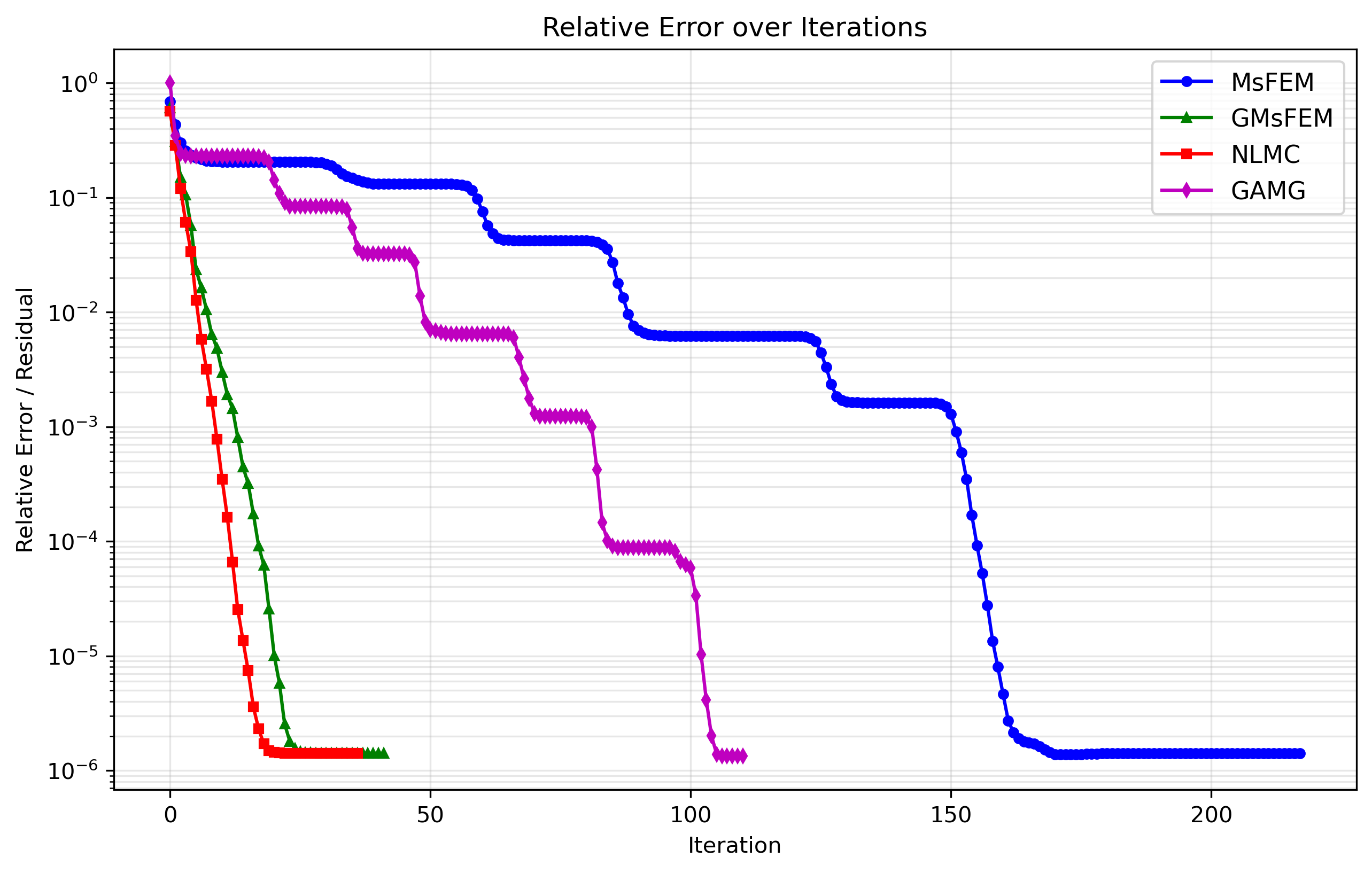}
\includegraphics[width=0.24\textwidth]{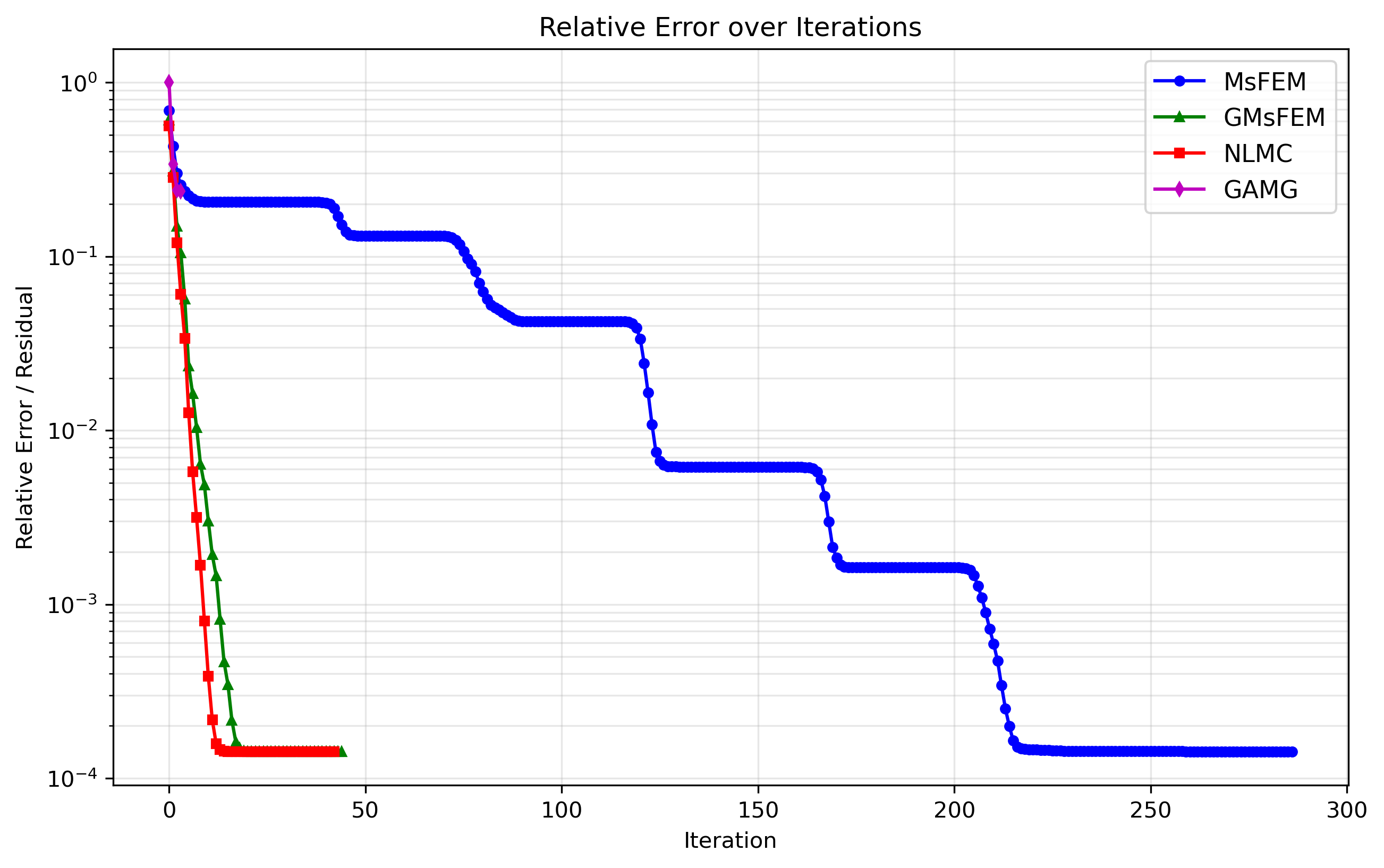}
\caption{The error history for different contrast ratio: $10^4$, $10^6$, $10^8$, $10^{10}$ from top left to bottom right. Here $NLMC$ refers to $V_{H,1}+V_{H,2}$. Time step size $dt=0.1$.}\label{fig:ex3_V12}
\end{figure}

\begin{figure}[ht!] \label{fig:ex3d_nf60_V1}
\centering
\includegraphics[width=0.24\textwidth]{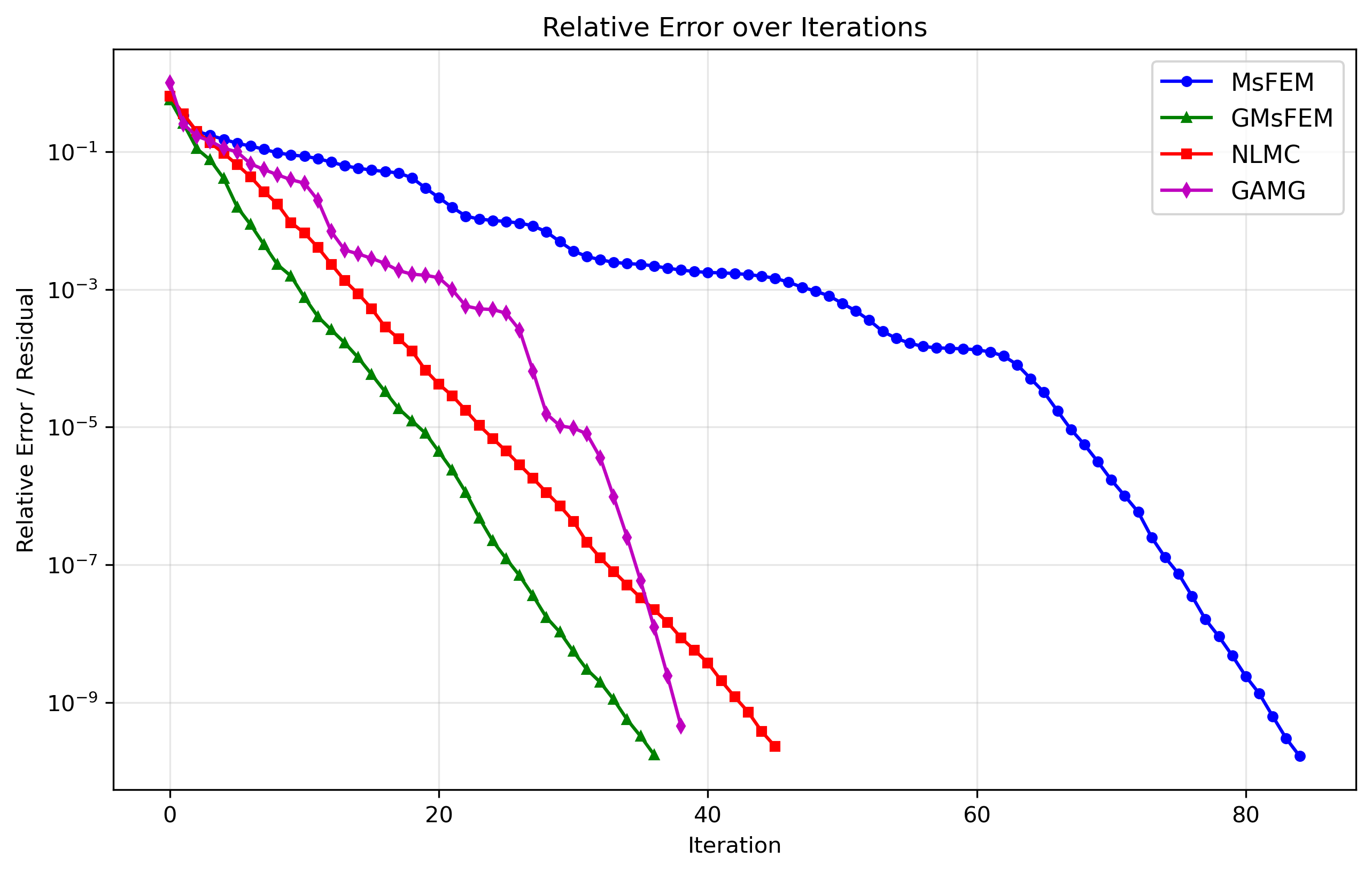}
\includegraphics[width=0.24\textwidth]{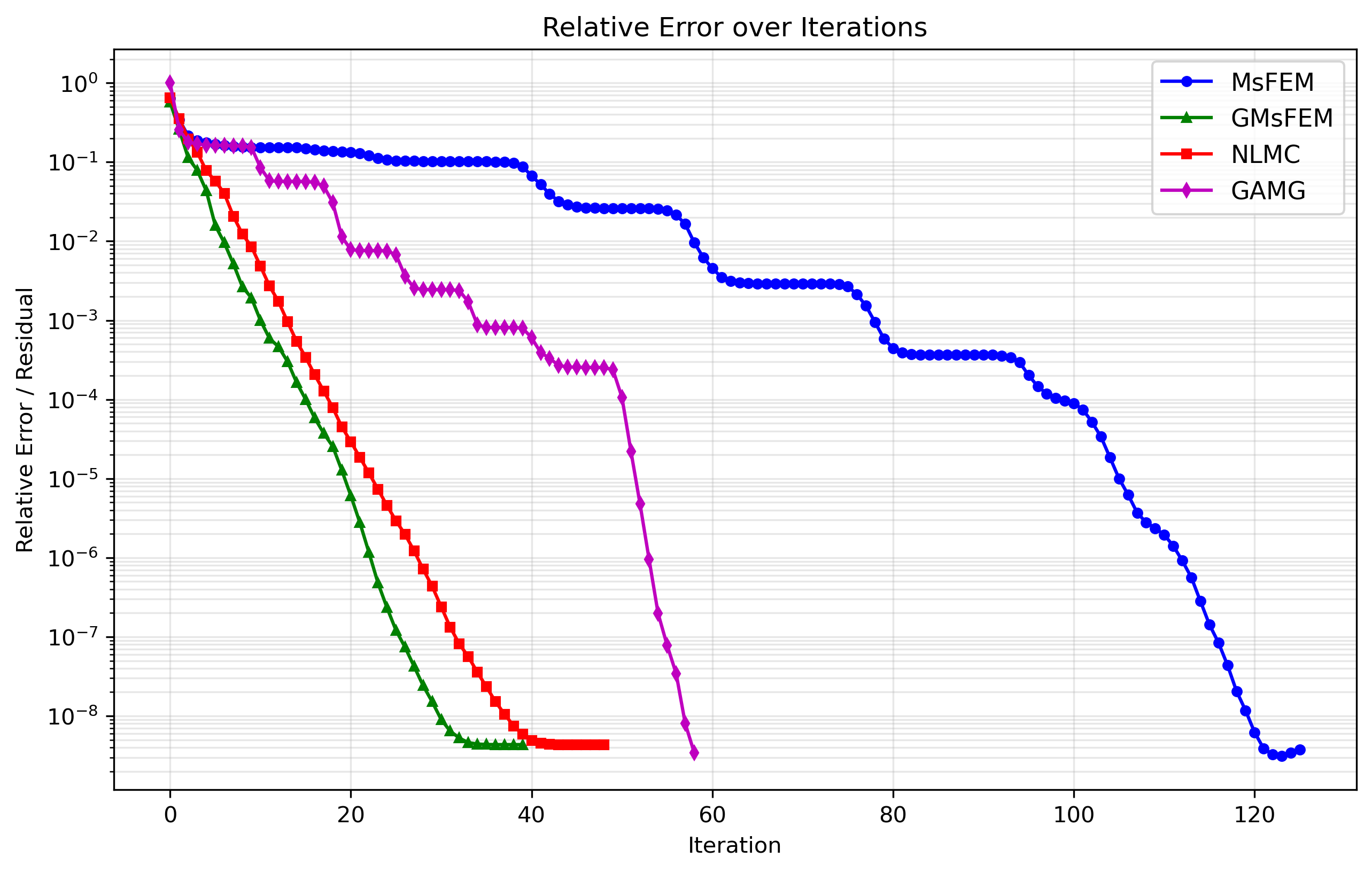}
\includegraphics[width=0.24\textwidth]{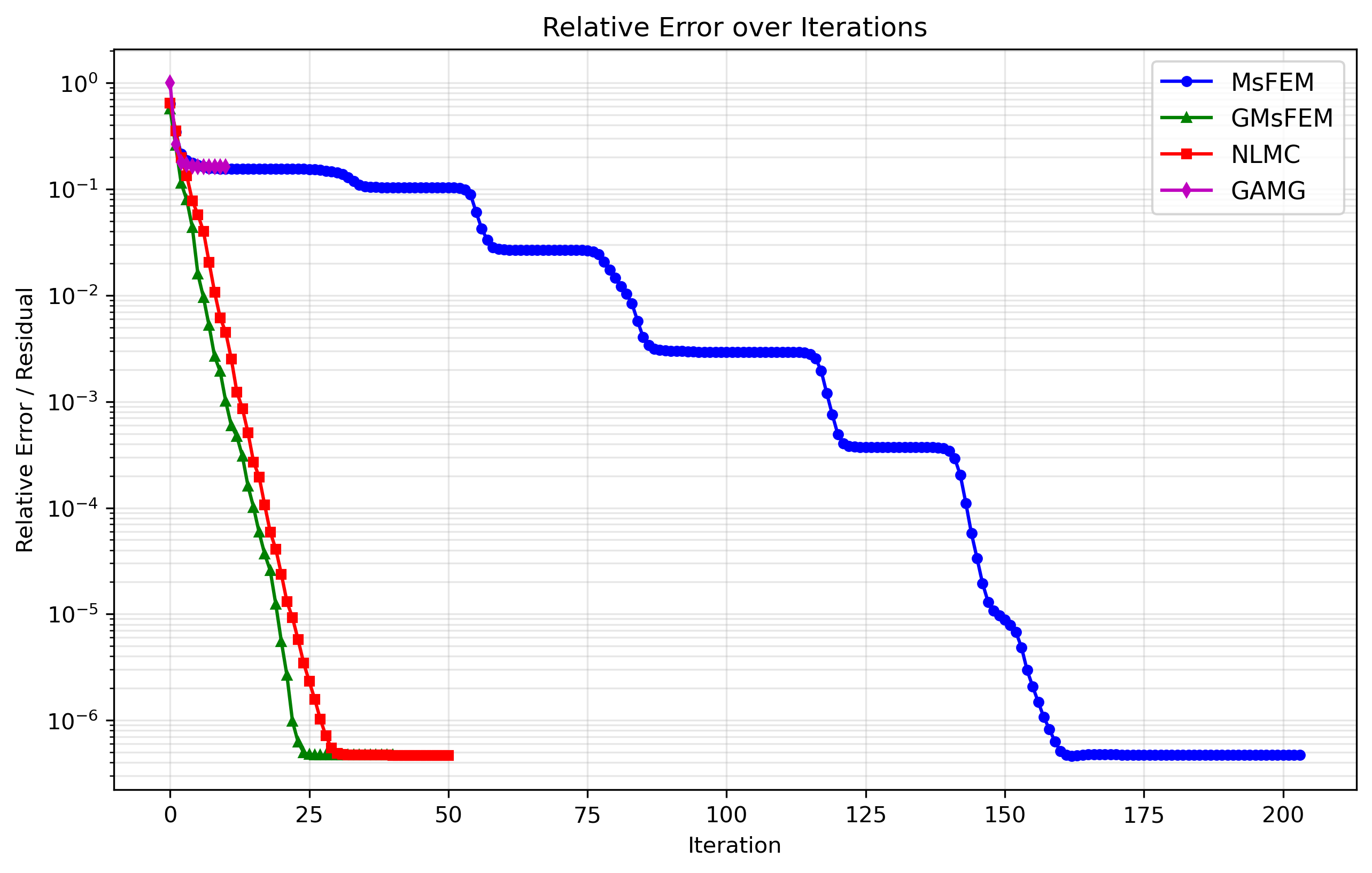}
\includegraphics[width=0.24\textwidth]{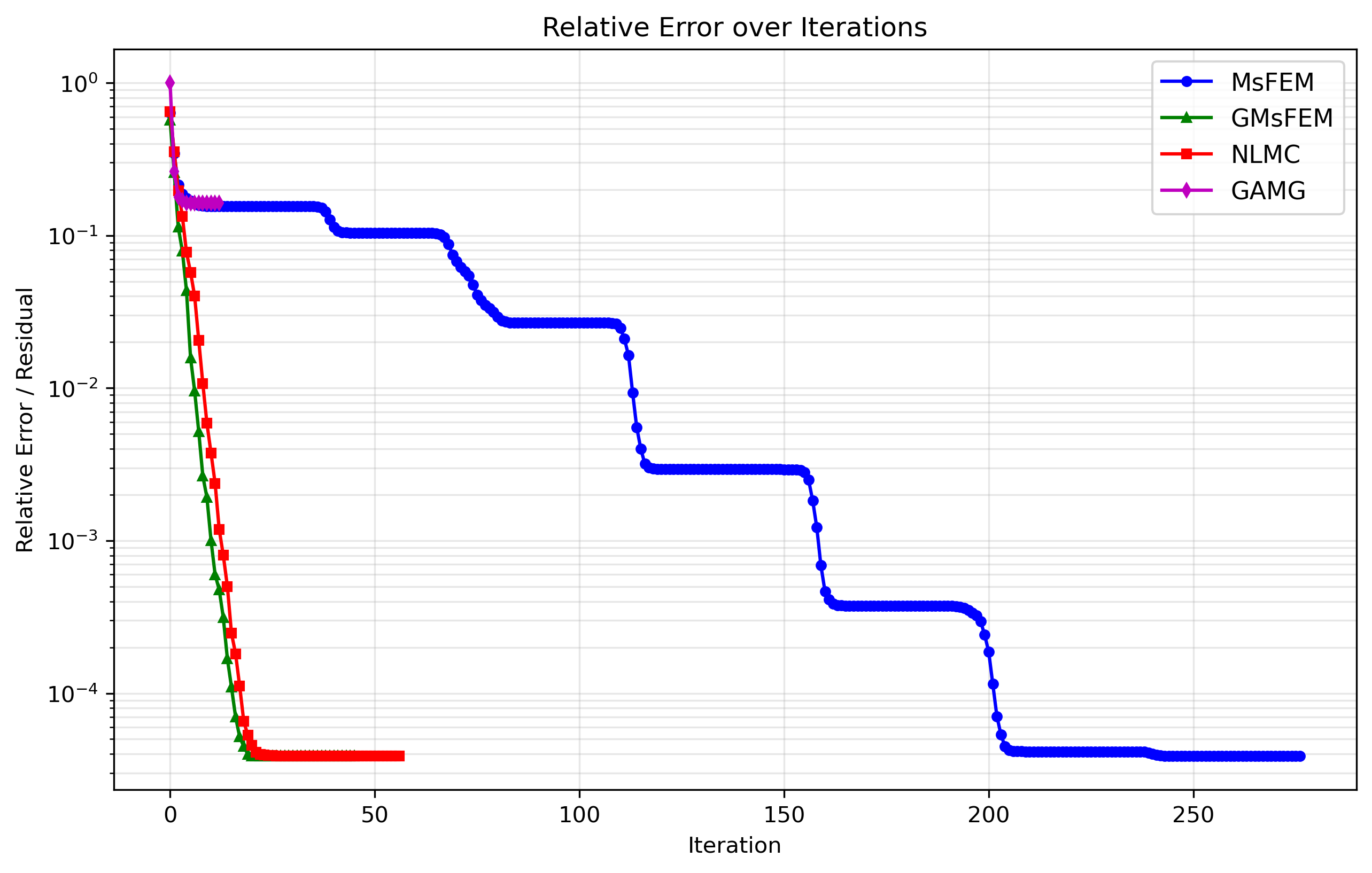}
\caption{The error history for different contrast ratio: $10^4$, $10^6$, $10^8$, $10^{10}$ from top left to bottom right. Here $NLMC$ refers to $V_{H,1}$. Time step size $dt=0.001$.}\label{fig:ex3_V1}
\end{figure}

\begin{table}[ht!]
  \centering
  \caption{CPU time (in seconds) and iteration counts for different methods and coarse ratios, where $V_H=V_{H,1}+V_{H,2}$ for CEM, $dt=0.1$. $H=1/10$, $h=1/60$, the fine Dofs is $226,981$.}
  \label{tab:pcg_iters_V12}
\begin{tabular}{@{}lcccc@{}}
    \toprule
    & $V_{\text{ms}}$ & $V_{\text{gms}}$ & $V_{H,1}+V_{H,2}$ & GAMG \\
    \cmidrule(lr){2-5}
    \diagbox[innerwidth=2.2cm,trim=l]{$cr$}{DOFs} & 1331 & 1209 & 1130 & --- \\
        \midrule
4  & 10.7 (89) & 4.6 (37) & 6.2 (32) & 3.0 (40) \\
6  & 15.9 (131) & 4.2 (34) & 5.2 (27) & 4.1 (60) \\
8  & 20.4 (169) & 3.1 (25) & 3.7 (19) & 7.4 (121) \\
10 & 26.1 (216) & 2.3 (19) & 2.5 (13) & -- \\
    \bottomrule
  \end{tabular}
\end{table}

\begin{table}[ht!]
  \centering
  \caption{CPU time (in seconds) and iteration counts for different methods and coarse ratios, where $V_H=V_{H,1}$ for CEM, $dt=0.001$. $H=1/10$, $h=1/60$, the fine Dofs is $226,981$.}
  \label{tab:pcg_iters_V1}
\begin{tabular}{@{}lcccc@{}}
    \toprule
    & $V_{\text{ms}}$ & $V_{\text{gms}}$ & $V_{H,1}$ & GAMG \\
    \cmidrule(lr){2-5}
    \diagbox[innerwidth=2.2cm,trim=l]{$cr$}{DOFs} & 1331 & 1209 & 130 & --- \\
    \midrule
4  & 10.1 (84) & 4.4 (36) & 8.7 (45) & 3.0 (40) \\
6  & 14.8 (122) & 4.1 (33) & 7.9 (41) & 3.9 (58) \\
8  & 19.4 (161) & 3.0 (24) & 5.9 (30) & -- \\
10 & 24.7 (205) & 2.3 (19) & 4.1 (21) & -- \\
    \bottomrule
  \end{tabular}
\end{table}

\begin{table}[ht!]
\centering
\caption{Iteration counts of two-level solvers for fixed fine mesh $h=1/60$, $cr=6$.}
\label{tab:strong_scalability}
\begin{tabular}{cccc}
\toprule
$H$  &$V_{\text{ms}}$ & $V_{\text{gms}}$  & $V_{H,1}+V_{H,2}$ \\
\midrule
 $1/10$ & $131$ & $34$ & $27$\\
 $1/12$ &  $131$ & $36$ & $28$\\
 $1/15$ & $127$ & $34$ & $26$ \\
\bottomrule
\end{tabular}
\end{table}

\begin{table}[ht!]
\centering
\caption{Iteration counts of two-level solvers for fixed ratio of $h/H=5$, $cr=4$.}
\label{tab:weak_scalability}
\begin{tabular}{cccccc}
\toprule
$H$ & $h$  & $V_{\text{ms}}$ & $V_{\text{gms}}$  & $V_{H,1}+V_{H,2}$ \\
\midrule
$1/10$ & $1/50$ & $92$ & $35$ & $32$ \\
$1/12$ & $1/60$ & $88$ & $33$ & $28$ \\
$1/14$ & $1/70$ & $86$ & $33$ & $32$ \\
\bottomrule
\end{tabular}
\end{table}

\section{Conlusion}
In this work, we developed a two-level overlapping preconditioner for fully implicit discretizations of time-dependent high-contrast multiscale problems. We introduced several ways to construct the coarse space. Especially, we proposed that the coarse space can be constructed within the NLMC framework which can separate the basis functions associated with connected high permeability networks from those representing the low permeability background. In addition, a relaxed energy-minimizing formulation is employed to construct the NLMC basis functions more efficiently by replacing exact continuum constraints with a projection-based penalization, and the relaxed energy-minimizing problem can be solved iteratively. We showed that the complete NLMC space provides an effective coarse space for preconditioning the high-contrast problem. For time-step sizes proportional to the square of the coarse-mesh size, the mass matrix controls the low-permeability contribution of the stiffness operator. Consequently, only the high-permeability features need to be treated carefully by the global coarse correction, leading to a substantially smaller coarse problem while preserving a condition-number bound independent of the coefficient contrast. For general time-step sizes, robustness can be retained by augmenting the high-permeability component with an efficiently constructed standard multiscale space. The proposed framework therefore combines heterogeneity-aware multiscale coarse spaces, overlapping domain decomposition, and the structure of the fully implicit transient operator to provide an efficient and robust solver for time-dependent high-contrast problems.

\section*{Acknowledgments}
YW's work is supported by the Tianyuan Fund for Mathematics of the National Natural Science Foundation of China (Grant No. 12526212), the National Natural Science Foundation of China (NSFC Grant No. 12301559).

\bibliographystyle{plain}
\bibliography{references}
\end{document}